\documentclass[12pt]{amsart}
\usepackage[margin=1in]{geometry}  
\usepackage{graphicx}              
\usepackage{amsmath}               
\usepackage{amsfonts}              
\usepackage{amsthm}                
\usepackage{caption} 
\usepackage{subcaption}
\usepackage{tikz}
\usepackage{enumitem}

\newtheorem{thm}{Theorem}[section]
\newtheorem{lem}[thm]{Lemma}

\newtheorem{prop}[thm]{Proposition}
\newtheorem{cor}[thm]{Corollary}

\theoremstyle{remark}
\newtheorem{remark}[thm]{Remark}

\theoremstyle{definition}
\newtheorem{definition}[thm]{Definition}

\begin{document}

\title{Tight contact structures on some plumbed 3-manifolds}
\author{Jonathan Simone}

\maketitle 

\begin{abstract}
In this article, we prove a generalization of Lisca-Matic's result in \cite{liscamatictightstrs} to Stein cobordisms and develop a method for distinguishing certain Stein cobordisms using rotation numbers. Using these results along with standard techniques from convex surface theory and classifications of tight contact structures on certain 3-manifolds due to Honda, we then classify the tight contact structures on a certain class of plumbed 3-manifolds that bound non-simply connected 4-manifolds. Moreover, we give descriptions of the Stein fillings of the Stein fillable contact structures.
\end{abstract}

\section{Introduction}\label{intro}

There are many classification results of tight contact structures on the boundaries of simply connected plumbings of $D^2-$bundles over $S^2$. For example, tight contact structures on small Seifert fibered spaces have been classified in \cite{ghigginiliscastipsicz}, \cite{ghigginischonenberger}, and \cite{wulegvertcirc}. These Seifert fibered spaces bound plumbings whose associated graphs are ``star-shaped". The proofs of these results are broken into two parts. First an upper bound $k$ for the number of tight contact structures is given using convex surface theory and applications of Honda's classifications of tight contact structures on the ``building blocks" $S^1\times D^2,$ $T^2\times I,$ and $S^1\times\Sigma$, where $\Sigma$ is a pair of pants (\cite{hondatight1}, \cite{hondatight2}). Then $k$ is shown to be a lower bound by exhibiting $k$ distinct Stein diagrams, which induce $k$ nonisotopic contact structures, by Lisca-Matic's result in \cite{liscamatictightstrs}.

This method clearly works well if the contact structures are Stein fillable. However, if the contact structures in question are not Stein fillable, then Lisca-Matic's result does not apply. By considering the Ozsv{\' a}th-Szab{\'o} contact invariant with $\omega$-twisted coefficients, we prove the following, which is a generalization of a result of Plamenevskaya in \cite{plamenevskaya}. 

\begin{thm} Suppose $(Y,\xi)$ is a contact manifold and $[\omega]\in H^2(Y;\mathbb{R})$ is an element such that $c(\xi,[\omega])$ is nontrivial. Let $(W,J_i)$ be a Stein cobordism from $(Y,\xi)$ to $(Y',\xi_i)$ for $i=1,2$. If the spin$^c$ structures induced by $J_1$ and $J_2$ are not isomorphic, then there exists an element $[\eta]\in H^2(Y';\mathbb{R})$ such that $c(\xi_1,[\eta])$ and $c(\xi_2,[\eta])$ are linearly independent.\label{cobordismthm}\end{thm}

One of the uses of this twisted coefficient system is that it can detect tight contact structures that the untwisted contact invariant does not detect, namely weakly symplectically fillable contact structures that are not strongly symplectically fillable (e.g. see \cite{ghigginiosfillability}). In particular, the following result is due to Ozsv{\' a}th and Szab{\'o}. 

\begin{thm}[Theorem 4.2 of \cite{OSgenusbounds}] If $(X,\omega)$ is a weak symplectic filling of $(Y,\xi)$, then the contact invariant $c(\xi;[\omega]|_Y)$ is non-trivial. \label{thmosfilling}\end{thm}

\noindent Coupling this result with Theorem \ref{cobordismthm}, we have the following corollary, which will be used in the proof of Theorem \ref{thm:fillablestrs}. This can be viewed as a generalization of Lisca-Matic's result in \cite{liscamatictightstrs}.

\begin{cor} If $(Y,\xi)$ is weakly symplectically fillable and $(W,J_i)$ is a Stein cobordism from $(Y,\xi)$ to $(Y',\xi_i)$ for $i=1,2$ such that the spin$^c$ structures induced by $J_1$ and $J_2$ are not isomorphic, then $\xi_1$ and $\xi_2$ are nonisotopic tight contact structures.\label{cobordismcor}\end{cor}

\begin{remark} In \cite{wuonlegsurg}, Wu proves a similar result. However, its application requires either a strong understanding of the topology of the weak filling or that $c^+(\xi)\neq0$. Corollary \ref{cobordismcor}, on the other hand, only requires $(Y,\xi)$ to be weakly fillable. Moreover by \cite{ghigginihondavanhornmorris}, $c^+$ vanishes in the cases with which we will be concerned.\end{remark}

It is known that the boundaries of non-simply connected plumbings (whose associated graphs are not trees) admit infinitely many tight contact structures due to the presence of incompressible tori (see \cite{hondakazezmatic}). Given a particular universally tight contact structure on a 3-manifold containing an incompressible torus $T$, one can add \textit{Giroux torsion} in a neighborhood of $T$ to produce infinitely many universally tight contact structures. We say a contact manifold contains \textit{Giroux $n-$torsion}, for $n\ge1$, if there exists a contact embedding of $(T^2\times I,\xi_n=\ker(\sin(2n\pi z)dx+\cos(2n\pi z)dy))$ into $(Y,\xi)$. Moreover, given an isotopy class $[T]$ of torus, the \textit{torsion} tor($Y,\xi,[T]$) is defined to be the supremum, over $n\in\mathbb{Z}^+$, such that there exists a contact embedding $\phi: (T^2\times[0,1], \xi_n) \to (Y,\xi)$, where $\phi(T^2\times\{pt\})\in[T]$. 

By a result of Gay in \cite{gayfillabletightstrs}, tight contact structures that are strongly symplectically fillable have \textit{no Giroux torsion} (i.e. there does not exist such an embedding). Thus, if one wishes to restrict to Stein fillable or strongly fillable contact structures, they may start by restricting their attention to tight contact structures with no Giroux torsion. It is also worth noting that, by \cite{colingirouxhonda2}, there are at most finitely many tight contact structures with no Giroux torsion on a given 3-manifold.

The graphs associated to non-simply connected plumbings contain \textit{cycles} of spheres (see Figure \ref{cycles}). In such plumbing graphs, each edge of each cycle must be decorated with either $``+"$ or $``-"$ to specify the sign of the intersection of the (oriented) base spheres (undecorated edges are understood to have sign ``+"). The boundaries of \textit{cyclic} plumbings, as depicted in Figure \ref{generalcyclicplumbing}, are $T^2-$bundles over $S^1$. In \cite{hondatight2}, Honda classified tight contact structures on such manifolds, many of which are parametrized by the amount of \textit{$S^1-$twisting}. The precise definition of $S^1-$twisting can be found in Section 0.0.1 in \cite{hondatight2}. We will briefly recall this definition (and the definition of the related notion of $I-$twisting) in Section \ref{twist}. Combining this work of Honda \cite{hondatight2} with work of Golla-Lisca \cite{gollalisca} and Ding-Geiges \cite{dinggeiges}, we have the following theorem. 

\begin{thm}[\cite{hondatight2}, \cite{gollalisca}, \cite{dinggeiges}] Let $C_{\pm}$ be the boundary of the cyclic plumbing depicted in Figure \ref{generalcyclicplumbing}, where $a_i\ge2$ for all $i$ and $a_1\ge3$. Then, up to isotopy, the tight contact structures on $C_{\pm}$ are completely classified as below.
\begin{itemize}
\item $C_+$ admits exactly $(a_1-1)\cdot\cdot\cdot(a_n-1)$ tight contact structures with no Giroux torsion, all of which are Stein fillable. For each $l\in\mathbb{Z}^+$, $C_+$ admits a unique universally tight, weakly fillable contact structure with $S^1-$twisting $2l\pi$.
\item $C_-$ admits exactly $(a_1-1)\cdot\cdot\cdot(a_n-1)$ virtually overtwisted tight contact structures and a unique universally tight contact structure with no Giroux torsion. The virtually overtwisted contact structures are all Stein fillable and the universally tight contact structure is Stein fillable if $(a_1,...,a_n)$ is embeddable (as defined in \cite{gollalisca}). For each $l\in\mathbb{Z}^+$, $C_-$ admits a unique universally tight, weakly fillable contact structure with $S^1-$twisting $(2l-1)\pi$.

\end{itemize}
\label{thm:hondacyclicthm}\end{thm}

Let $Y_{\pm}$ be the plumbed 3-manifold obtained as the boundary of the plumbing depicted in Figure \ref{generalsingularplumbing}. The main result of this paper is the following theorem. The notion of twisting mentioned in this theorem will be defined in Section \ref{sfs}.

\begin{figure}
\centering
\begin{subfigure}{.4\textwidth}
\centering
\includegraphics[scale=.6]{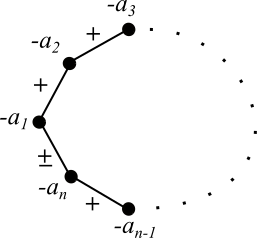}
\caption{A cyclic plumbing with boundary $C_{\pm}$}\label{generalcyclicplumbing}
\end{subfigure}
\begin{subfigure}{.4\textwidth}
\includegraphics[scale=.6]{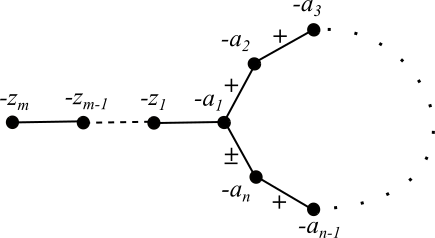}
\caption{A cyclic plumbing with an arm with boundary $Y_{\pm}$}\label{generalsingularplumbing}
\end{subfigure}
\caption{}\label{cycles}\end{figure}

\begin{thm} Let $a_i,z_j\ge2$ for all $i,j$ and $a_1\ge3$. Then, up to isotopy, the tight contact structures on $Y_{\pm}$ are completely classified as below.
\begin{itemize}
\item $Y_+$ admits exactly $(a_1-1)\cdot\cdot\cdot(a_n-1)(z_1-1)\cdot\cdot\cdot(z_m-1)$ tight contact structures with no Giroux torsion, all of which are Stein fillable. For each $l\in\mathbb{Z}^+$, $Y_+$ admits exactly $z_1(z_2-1)\cdot\cdot\cdot(z_m-1)$ weakly fillable contact structures with twisting $2l\pi$.
\item $Y_-$ admits exactly $(a_1-1)\cdot\cdot\cdot(a_n-1)(z_1-1)\cdot\cdot\cdot(z_m-1)+z_1(z_2-1)\cdot\cdot\cdot(z_m-1)$ tight contact structures with no Giroux torsion. $(a_1-1)\cdot\cdot\cdot(a_n-1)(z_1-1)\cdot\cdot\cdot(z_m-1)$ of them are Stein fillable and if $(a_1,...,a_n)$ is embeddable, then all of these contact structures are Stein fillable. For each $l\in\mathbb{Z}^+$, $Y_-$ admits exactly $z_1(z_2-1)\cdot\cdot\cdot(z_m-1)$ weakly fillable contact structures with twisting $(2l-1)\pi$.

\end{itemize}
 \label{thm:fillablestrs}\end{thm}

\begin{remark} The proof of Theorem \ref{thm:fillablestrs} can be modified to classify the tight contact structures for $Y_{\pm}$ in more general settings. That is, one can remove the assumption $a_1\ge 3$ in certain cases and prove analogous results. \end{remark}

The proof of Theorem \ref{thm:fillablestrs} for the Stein fillable contact structures is fairly standard. It relies on convex surface theory to provide an upper bound for the number of tight contact structures and then by producing explicit Stein fillings with distinct first Chern classes, we realize this upper bound. In the non-Stein fillable cases (e.g. when Giroux torsion is present), we will appeal to Theorem \ref{cobordismthm}, which will be proved in Section \ref{invt}, and to the discussion in Section \ref{steincobordism}, which presents a method (analogous to Proposition 2.3 in \cite{gompfsteinsurfaces}) of distinguishing almost complex structures on Stein cobordisms using rotation numbers. Section \ref{cst} contains an overview of relevant convex surface theory results of Giroux and Honda and the proof of Theorem \ref{thm:fillablestrs} is found in Section \ref{fillable}. Finally, in Section \ref{Y_-filling}, we provide explicit descriptions of the Stein fillings of the tight contact structures on $Y_-$ satisfying the condition: ``$(a_1,...,a_n)$ is embeddable."\\

\noindent\textbf{Acknowledgements} The author would like to thank his advisor, Tom Mark, for his encouragement, patience, and guidance, and Bulent Tosun for his helpful input and enthusiasm.

\section{The Ozsv{\' a}th-Szab{\'o} contact invariant with $\omega$-twisted coefficients}\label{invt}

In this section, we will recall the definition of the contact invariant with $\omega$-twisted coefficients, as defined in \cite{OSgenusbounds}, and use it to prove Theorem \ref{cobordismthm} below, which will in turn be used in the proof of Theorem \ref{thm:fillablestrs}. We will assume the reader is familiar with Heegaard Floer homology with twisted coefficients and the contact invariant (see \cite{heegfloer2}, \cite{heegfloerMaps}).

Let $Y$ be a three-manifold and fix a cohomology class $[\omega]\in H^2(Y;\mathbb{R})$. We can then view $\mathbb{Z}[\mathbb{R}]$ as a $\mathbb{Z}[H^1(Y;\mathbb{Z})]$-module via the ring homomorphism $[\gamma]\to T^{\langle \gamma\cup\omega,[Y]\rangle}$, where $T^r$ denotes the group ring element associated to $r\in\mathbb{R}$. Using this coefficient system, we denote the \textit{$\omega$-twisted Floer homology} by $\widehat{\underline{HF}}(Y;[\omega])$. Let $W:Y\to Y'$ be a cobordism and let $[\omega]\in H^2(W;\mathbb{R})$. Then for each $\mathfrak{s}\in$Spin$^c(W)$, we obtain an induced map $\underline{F}_{W,\mathfrak{s};[\omega]}:\widehat{\underline{HF}}(Y,\mathfrak{s}|_{Y};[\omega]|_{Y})\to\widehat{\underline{HF}}(Y',\mathfrak{s}|_{Y'};[\omega]|_{Y'})$, which is well-defined up to multiplication by $\pm T^c$ for some $c\in \mathbb{R}$. See \cite{OSgenusbounds} for more details.

Given a contact structure $\xi$ on $Y$, we can define the \textit{$\omega$-twisted contact invariant} $c(Y;[\omega])\in\underline{\widehat{HF}}(-Y,\mathfrak{t_{\xi}};[\omega])$, where $\mathfrak{t_{\xi}}$ denotes the canonical spin$^c$ structure on $Y$ determined by $\xi$. This element is well-defined up to sign and multiplication by invertible elements in $\mathbb{Z}[H^1(Y;\mathbb{Z})]$. We denote its equivalence class by $[c(\xi;[\omega])]$.

The following theorem follows by the proof of Theorem 3.6 in \cite{ghigginivhmorris}. We will use it to prove the main result of this section, Theorem \ref{cobordismthm} below, which can be thought of as a generalization of a result due to Plamenevskaya in \cite{plamenevskaya}.

\begin{thm} Let $(Y,\xi)$ and $(Y',\xi')$ be contact manifolds and let $(W,J)$ be a Stein cobordism from $(Y,\xi)$ to $(Y',\xi')$ which is obtained by Legendrian surgery on some Legendrian link in $Y$. If $\mathfrak{t}$ is the canonical spin$^c$ structure on $W$ for the complex structure $J$, then: $$[F_{\overline{W},\mathfrak{s};[\omega]}(c(\xi';[\omega]|_{Y'}))]=  \left\{
\begin{array}{ll}
      [c(\xi;[\omega]|_{Y})] \text{  if  } \mathfrak{s}=\mathfrak{t} \\
      0 \text{  if  } \mathfrak{s}\neq\mathfrak{t}
\end{array} 
\right.$$
\label{ghigginivhmorristhm}\end{thm}

\noindent\textbf{Theorem \ref{cobordismthm}.} \textit{Suppose $(Y,\xi)$ is a contact manifold and $[\omega]\in H^2(Y;\mathbb{R})$ is an element such that $c(\xi,[\omega])$ is nontrivial. Let $(W,J_i)$ be a Stein cobordism from $(Y,\xi)$ to $(Y',\xi_i)$ for $i=1,2$. If the spin$^c$ structures induced by $J_1$ and $J_2$ are not isomorphic, then there exists an element $[\eta]\in H^2(Y';\mathbb{R})$ such that $c(\xi_1,[\eta])$ and $c(\xi_2,[\eta])$ are linearly independent.}

\begin{proof} 
Since $W$ has no 3-handles, $H^3(W,Y)=0$ and so there exists an element $[\Omega]\in H^2(W;\mathbb{R})$ satisfying $[\Omega]|_Y=[\omega]$. Let $\mathfrak{s}_1,\mathfrak{s}_2\in \text{Spin}^c(W)$ such that $\mathfrak{s}_i|_{Y}=\mathfrak{t}_{\xi}$ and $\mathfrak{s}_i|_{Y'}=\mathfrak{t}_{\xi_i}$ for $i=1,2$. Consider the cobordism maps $F_{\overline{W},\mathfrak{s}_i;[\Omega]}:\underline{\widehat{HF}}(-Y',\mathfrak{t}_{\xi_i};[\Omega]|_{Y'})\to\underline{\widehat{HF}}(-Y,\mathfrak{t}_{\xi};[\Omega]|_{Y})$. By Theorem \ref{ghigginivhmorristhm}, $[F_{\overline{W},\mathfrak{s}_i;[\Omega]}(c(\xi_i;[\Omega]|_{Y'}))]=[c(\xi;[\Omega]|_{Y})]$ if $\mathfrak{s_i}=\mathfrak{t_i}$, where  $\mathfrak{t_i}$ is the canonical spin$^c$ structure associated to $J_i$. Thus $[c(\xi_1;[\Omega]|_{Y'})]$ and $[c(\xi_2;[\Omega]|_{Y'})]$ are both nontrivial. Moreover, $[F_{\overline{W},\mathfrak{s}_i;[\Omega]}(c(\xi_i;[\Omega]|_{Y'}))]=0$ whenever $\mathfrak{s}_i\neq\mathfrak{t}_i$. In particular, $[F_{\overline{W},\mathfrak{t}_i;[\Omega]}(c(\xi_j;[\Omega]|_{Y'}))]=0$ when $i\neq j$. Thus, since $\mathfrak{t}_1\neq\mathfrak{t}_2$, we have that $[c(\xi_1;[\Omega]|_{Y'})]\neq [c(\xi_2;[\Omega]|_{Y'})]$. Moreover, the contact elements $c(\xi_1;[\Omega]|_{Y'})$ and $c(\xi_2;[\Omega]|_{Y'})$ live in different summands of $\underline{\widehat{HF}}(-Y;[\Omega]|_{Y})$ and are thus linearly independent.\end{proof}

\section{Legendrian surgery in $T^2\times I$}\label{steincobordism}

In this section, we will describe a method to distinguish contact structures obtained by Legendrian surgery on 3-manifolds containing a particular contact $T^2\times[0,1]$. This method will be used in the proof of Theorem \ref{thm:fillablestrs} found in Section \ref{fillable}. Give $T^2\times [0,1]$ the coordinates $((x,y),t)$ and define the contact structure $\xi$ on $T^2\times[0,1]$ to be the kernel of the 1-form $\alpha=\sin(\phi(t))dx+\cos(\phi(t))dy$, where $\phi'(t)>0$, $\phi(0)=-\frac{\pi}{2}$, and $\phi(1)=\frac{\pi}{2}$. Then there exists a torus $T_{t_0}=T^2\times\{t_0\}$, such that $\phi(t_0)=0$ so that the contact form restricted to $T_{t_0}$ is $\alpha=dy$.

Consider the standard diagram of $T^2\times[0,1]$ as embedded in $\mathbb{R}^3$ depicted in Figure \ref{surgerycurve3} (without the red surgery curves), where $T_1=T^2\times\{1\}$ is the outer torus and $T_0=T^2\times\{0\}$ is the inner torus. Moreover, let $S^1\times \{pt\}\times\{pt\}$ be the longitudinal direction and $\{pt\}\times S^1\times\{pt\}$ be the meridional direction in this diagram. Then we can draw $T_{t_0}$ as a square, with its edges identified, such that the horizontal edges of the square are the $x$-direction and the vertical edges of the square are the $y$-direction. Let $\gamma=S^1\times\{pt\}\times\{t_0\}$ and define the 0-framing associated to $\gamma$ to be surface framing of $\gamma$ in the surface $T_{t_0}$. Denote this framing by $\mathcal{F}$. Then any knot smoothly isotopic to $\gamma$ has a well-defined 0-framing, namely the image of $\mathcal{F}$ under the isotopy. For any nullhomologous knot in $T^2\times [0,1]$, the 0-framing is given by the Seifert surface framing.

\begin{figure}[h]
\centering
\begin{subfigure}{.45\textwidth}
\centering
\includegraphics[scale=.55]{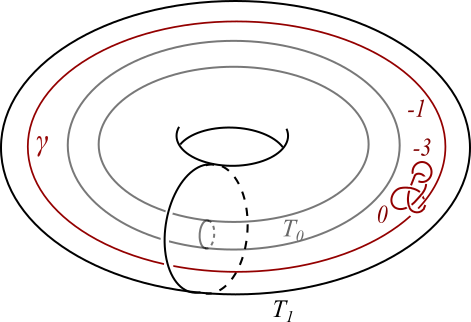}
\caption{A surgery link $L$ in $T^2\times[0,1]$}\label{surgerycurve3}
\end{subfigure}
\begin{subfigure}{.45\textwidth}
\centering
\includegraphics[scale=.55]{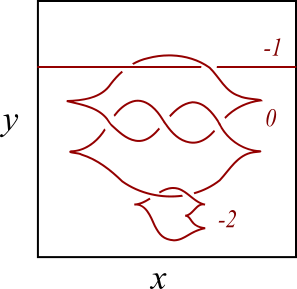}
\caption{The front projection of $L$. The framings are smooth framings.}\label{frontprojection3}
\end{subfigure}
\caption{}
\end{figure}

As in the case of Legendrian knots in $(\mathbb{R}^3,\xi_{st})$, we can project any Legendrian curve $L\subset(T^2\times(0,1),\xi)$ to $T_{t_0}$. We call this the \textit{front projection} of $L$. If $L\subset T^2\times(0,1)$, then the projection will have no vertical tangencies, since $\frac{dy}{dx}=-\tan(\phi(t))\neq\infty$ for all $t\in(0,1)$. It will, however, contain semi-cubical cusps and away from these cusp points $L$ can be recovered by $\frac{dy}{dx}=-\tan(\phi(t))$. In particular, at a crossing the strand with smaller slope is in front. For example, Figure \ref{frontprojection3} shows a front projection of the link depicted in Figure \ref{surgerycurve3}. We will only concern ourselves with nullhomologous knots that can be contained in a 3-ball and knots that are smoothly isotopic to $\gamma$.

Give $\mathbb{R}^3$ the coordinates $(u,v,w)$ so that $\xi_{st}=\text{ker}(dw+udv)$ and let $\tilde{\xi}_{st}$ be the image of $\xi_{st}$ under the projection $\mathbb{R}^3\to \mathbb{R}\times(\mathbb{R}^2/\mathbb{Z}^2)\cong (0,1)\times T^2$. It is easy to see that $(T^2\times (0,1),\xi)$ is isotopic to $(T^2\times (0,1),\tilde{\xi}_{st})$. In particular, the contact planes of $\tilde{\xi}_{st}$ and $\xi$ twist in similar fashions. Thus, for a front projection $K$ of a nullhomologous Legendrian knot that can be contained in a 3-ball, the Thurston-Bennequin number $tb(K)$ and the rotation number $r(K)$ can be defined and computed in the same way for Legendrian knots in $(\mathbb{R}^3,\xi_{st})$. That is, $tb(K)=w(K)-\frac{1}{2}c(K)$ and $r(K)=\frac{1}{2}(c_d(K)-c_u(K))$, where $w(K)$ is the writhe of $K$, $c(K)$ is the total number of cusps of $K$, $c_d(K)$ is the number of ``down" cusps of $K$, and $c_u(K)$ is the number of ``up" cusps of $K$. Now let $\tilde{\gamma}$ be a Legendrian knot that is smoothly isotopic to $\gamma$ (and is thus not nullhomologous). As above, the twisting number along $\tilde{\gamma}$ with respect to $\mathcal{F}$, which we denote by $tb(\tilde{\gamma},\mathcal{F})$, can also be computed using the formula $tb(\tilde{\gamma},\mathcal{F})=w(\tilde{\gamma})-\frac{1}{2}c(\tilde{\gamma})$. For simplicity, we will drop the decoration $\mathcal{F}$ and simply write $tb(\tilde{\gamma})$. Next, since $\frac{\partial}{\partial t}\in\xi$ is a nonvanishing vector field, we can define the rotation number of $\tilde{\gamma}$ with respect to $\frac{\partial}{\partial t}$, denoted by $r_{\partial/\partial t}(\tilde{\gamma})$, to be the signed number of times that the tangent vector field to $\tilde{\gamma}$ rotates in $\xi$ relative to $\frac{\partial}{\partial t}$ as we traverse $\tilde{\gamma}$. For simplicity, we will write $r(\tilde{\gamma})=r_{\partial/\partial t}(\tilde{\gamma})$. It is once again easy to see that we can compute $r(\tilde{\gamma})$ using the formula $r(\tilde{\gamma})=\frac{1}{2}(c_d(\tilde{\gamma})-c_u(\tilde{\gamma}))$. In particular, for the Legendrian knot $\gamma$, we have that $tb(\gamma)=0=r(\gamma)$.

Now suppose $T^2\times[0,1]$ is embedded in a closed tight contact 3-manifold $(Y,\xi)$ such that $c_1(\xi)=0$ and $\xi|_{T^2\times[0,1]}$ is isotopic to the contact structure above. Further suppose $(W,J)$ is a Stein cobordism from $(Y,\xi)$ to $(Y',\xi')$ obtained by attaching 2-handles $\{h_i\}_{i=1}^n$ along Legendrian knots $\{K_i\}_{i=1}^n$, where each $K_i$ is either contained in a 3-ball or is smoothly isotopic to $\gamma$ in $T^2\times(0,1)$. Moverover, suppose these knots have respective smooth framings $\{tb(K_i)-1\}_{i=1}^n$. Assume we can extend $\frac{\partial}{\partial t}$ to a nonvanishing vector field $v\in\xi$ (which trivializes $\xi$ as a 2-plane bundle). Let $w\in TW|_{Y}$ be an outward normal vector field to $Y$. Then the frame $(v, Jv, Jw)$ gives a trivialization $\tau$ of $TY$. Following the arguments of Proposition 2.3 in \cite{gompfsteinsurfaces}, we prove the following.

\begin{lem} $c_1(W,J,\tau)$ can be represented by a cocycle whose value on $h_i$ is equal to $r(K_i)$.\label{chernlemma}\end{lem}

\begin{proof} By \cite{eliashbergcomplexif}, we can thicken $Y$ to a Stein cobordism $Y\times[0,1]$ from $(Y,\xi)$ to itself. We can extend $\tau$ of $TY$ to a complex trivialization of $T(Y\times[0,1])$ using the inward pointing normal vector field $-\frac{\partial}{\partial s}$ (which agrees with $w$ on $Y$), where $s$ is the coordinate on $[0,1]$. To form $W$, we attach the 2-handles $h_i$ to $Y\times\{1\}$. By definition, $c_1(W,J,\tau)$ measures the failure to extend the trivialization of $T(Y\times[0,1])$ over $h_i$ for all $i$. For each $i$, viewing $h_i\cong D^2\times D^2\subset i\mathbb{R}^2\times\mathbb{R}^2$, we can build a complex trivialization of $Th_i$. First trivialize $T(D^2\times0)|_{\partial D^2}$ by using the tangent vector field $a$ to $\partial D^2$ and the outward normal vector field $b$. We can then extend this trivialization to a complex trivialization $(a^*,b^*)$ of $Th_i$ (see \cite{gompfsteinsurfaces} for details). Now, when we attach $h_i$ to $Y$, $a$ is identified with a tangent vector field to $K_i$ and $b$ is identified with $-\frac{\partial}{\partial s}|_{K_i}$. Thus $a^*$ and $v$ both span $\xi$ when restricted to $TY$ and thus together they span a complex line bundle $L_1$ on $(Y\times I)\cup W$. Moreover, $b^*$ and $-\frac{\partial}{\partial s}$ fit together to span a complementary trivial line bundle $L_2$. Since $T((Y\times I)\cup W)=L_1\oplus L_2$, the cochain associated to $c_1(W,J,\tau)$ evaluated on $h_i$ is clearly given by the rotation number of $a$ in $\xi$ relative to $\frac{\partial}{\partial t}$.\end{proof}

We will use Lemma \ref{chernlemma} in the following context. Suppose 2-handles are attached along a link $L=K_1\sqcup\cdot\cdot\cdot\sqcup K_n\subset T^2\times (0,1)\subset Y$ with respective framings $-a_i$ to obtain $Y'$, where each $K_i$ is either contained in a 3-ball or is smoothly isotopic to $\gamma$. Further supposed that there exists a front projection $L'$ of $L$ such that $tb(K'_i)\ge -a_i+1$ for all $i$. Then for each $i$, we can stabilize $K'_i$ $(tb(K'_i)+a_i-1)$-times to obtain a Legendrian knot satisfying $tb(K'_i)=-a_i+1$. There are two kinds of stabilizations (i.e with an upward cusp or a downward cusp), which affect the rotation numbers differently. Thus, for each $i$, there are $tb(K'_i)+a_i-1$ different stabilizations possible for $K'_i$. As a quick example, notice that the link in Figure \ref{frontprojection3} has two stabilization possibilities. Now by Lemma \ref{chernlemma}, these different kinds of stabilizations yield distinct Stein cobordisms. Moreover, if the hypothesis of Theorem \ref{cobordismthm} is satisfied, then the induced contact structures on $Y'$ are nonisotopic.

\section{Results from convex surface theory}\label{cst}

We will assume that the reader is familiar with convex surface theory due to Giroux \cite{girouxconvex} and we will list some key results about bypass attachments due to Honda \cite{hondatight1} which will be used throughout the rest of the paper. For a nice exposition on the basics of convex surface theory, see \cite{ghigginischonenberger}. First recall that, by Giroux \cite{girouxconvex}, any embedded orientable surface $\Sigma$ (that is either closed or has Legendrian boundary with nonpositive twisting number) in a contact 3-manifold can be perturbed to be \textit{convex}. This is equivalent to the existence of a collection of curves $\Gamma_{\Sigma}\subset\Sigma$ called the \textit{dividing set} that satisfies certain properties (see \cite{ghigginischonenberger}). If $T^2$ is a convex torus, then by Giroux's criterion (Theorem 3.1 in \cite{hondatight1}), $\Gamma_{T^2}$ consists of (an even number of) parallel dividing curves. Identifying $T^2$ with $\mathbb{R}^2/\mathbb{Z}^2$, the slope $s$ of the dividing curves is called the \textit{boundary slope} and denoted by $s(\Gamma_{T^2})$. By Giroux's Flexibility Theorem in \cite{girouxconvex}, $T^2$ can be further perturbed (relative to $\Gamma_{T^2}$) so that the characteristic foliation consists of a 1-parameter family of closed curves called \textit{Legendrian rulings}. Each of these curves has the same slope $r$, called the \textit{ruling slope}. In this case, each component of $T^2\backslash\Gamma_{T^2}$ contains a line of singular points of slope $s$ called a \textit{Legendrian divide}. A convex torus that is in this form is said to be in \textit{standard form}.

\begin{thm}[Flexibility of Legendrian rulings \cite{hondatight1}] Assume $T^2$ is a convex torus in standard form, and, using $\mathbb{R}^2/\mathbb{Z}^2$ coordinates, has boundary slope $s$ and ruling slope $r$. Then by a $C^0$-small perturbation near the Legendrian divides, we can modify the ruling slope from $r\neq s$ to any other $r'\neq s$ (including $\infty$).\label{thm:flexibility}\end{thm}

\begin{prop}[\cite{hondatight1}] Assume $T^2\times I$ has convex boundary in standard form and the boundary slope on $T^2\times \{i\}$ is $s_i$ for $i=0,1$. Then, we can find convex tori parallel to $T^2\times\{0\}$ with any boundary slope $s$ in $[s_1,s_0]$ (including $\infty$ if $s_0<s_1$).\label{prop:parallelslopes}\end{prop}

\begin{thm}[The Farey Tessellation \cite{hondatight1}] Assume $T$ is a convex torus in standard form with $\#\Gamma_T=2$ and boundary slope $s$. If a bypass is attached along a Legendrian ruling curve of slope $r\neq s$ to the ``front" of $T$, then the resulting convex torus $T'$ will have $\#\Gamma_{T'}=2$ and its boundary slope $s'$ is obtained from the Farey tessellation as follows. Let $[r,s]$ be the arc on $\partial\mathbb{D}$ (where $\mathbb{D}$ is the disc model of the hyperbolic plane) running from $r$ to $s$ counterclockwise. Then $s'$ is the point in $[r,s]$ closest to $r$ with an edge to $s$. If the bypass is attached to the ``back" of $T$, then we use the same algorithm except we use the interval $[s,r]$.\label{thm:fareytess} \end{thm}

\begin{thm}[The Imbalance Principle \cite{hondatight1}] Suppose $\Sigma$ and $\Sigma'$ are two disjoint convex surfaces and let $A$ be a convex annulus whose interior is disjoint from both $\Sigma$ and $\Sigma'$ and whose boundary is Legendrian with one component on each surface. If $|\Gamma_{\Sigma}\cdot\partial A|>|\Gamma_{\Sigma'}\cdot\partial A|$, then by the Giroux Flexibility Theorem \cite{girouxconvex}, there exists a bypass for $\Sigma$ on $A$.\label{thm:imbalance}\end{thm}

\begin{lem}[The Edge Rounding Lemma \cite{hondatight1}] Let $\Sigma_1$ and $\Sigma_2$ be convex surfaces with collared Legendrian boundaries which intersect transversely inside an ambient contact manifold along a common boundary Legendrian curve. Assume the neighborhood of the common boundary Legendrian is locally isomorphic to the neighborhood $N_{\epsilon}=\{x^2+y^2\le\epsilon\}$ of $M = \mathbb{R}^2\times(\mathbb{R}/\mathbb{Z})$ with coordinates $(x, y, z)$ and contact 1-form $\alpha= \sin(2\pi nz)dx + \cos(2\pi nz)dy$, for some $n \in\mathbb{Z}^+$, and that $\Sigma_1\cap N_{\epsilon}=\{x=0,0\le y\le \epsilon\}$ and $\Sigma_2\cap N_{\epsilon}=\{y=0,0\le x\le \epsilon\}$. If we join $\Sigma_1$ and $\Sigma_2$ along $x = y = 0$ and round the common edge so that the orientations of $\Sigma_1$ and $\Sigma_2$ are compatible and
induce the same orientation after rounding, the resulting surface is convex, and the dividing curve $z=\frac{k}{2n}$ on $\Sigma_1$ will connect to the dividing curve $z=\frac{k}{2n}-\frac{1}{4n}$ on $\Sigma_2$, where $k=0,1,\cdot\cdot\cdot,2n-1$.\label{lem:edgeround}\end{lem}

We will use these tools in the following context. Let $\Sigma$ be a pair of pants and consider a contact 3-manifold $S^1\times \Sigma$. Identify each boundary component of $-\partial (S^1\times \Sigma)$ with $\mathbb{R}^2/\mathbb{Z}^2$ by setting $(0,1)^T$ as the direction of the $S^1-$fiber and $(1,0)^T$ as the direction given by $-\partial(\{pt\}\times\Sigma)$. Let $T_0$ and $T_1$ be convex tori isotopic to two different boundary components of $S^1\times\Sigma$ and suppose these tori have boundary slopes $\frac{b}{a}$ and $\frac{t}{s}$, respectively, where $a,s>0$. Moreover, assume both dividing sets have $2k$ curves. By Theorem \ref{thm:flexibility}, we can arrange that the Legendrian rulings on both tori have infinite slope. Suppose there exists a convex ``vertical" annulus $A$ whose boundary components lie on Legendrian rulings of each torus. If $a\neq s$, then by the Imbalance Principle, there will exist a bypass along either $T_0$ or $T_1$.  If $a=s$ then there will either exist a bypass along both $T_0$ and $T_1$ or there will be no bypasses. If there do exist bypasses and $k>1$, then attaching the bypasses decreases $k$ by 1, but leaves the boundary slope unchanged. If $k=1$, then attaching the bypasses decreases the boundary slopes as described in Theorem \ref{thm:fareytess}. If there do not exist bypasses, then we may use the Edge Rounding Lemma four times to produce a new torus $T$ made up of $T_0, T_1$, and two parallel copies of $A$. Notice that $T$ now contains exactly 2 dividing curves, each of which wraps around $T$ $(kb+kt+1)$-times in the $S^1-$direction and $ka$-times in the $-\partial(\{pt\}\times\Sigma)$-direction. Thus the boundary slope of $T$ is $\frac{kb+kt+1}{ka}=\frac{b}{a}+\frac{t}{a}+\frac{1}{ka}$.

\subsection{Twisting}\label{twist} 
Consider a tight contact structure $\xi$ on $T^2\times [0,1]$ with $s(\Gamma_{T^2\times\{i\}})=s_i$ for $i=0,1$. $\xi$ is called \textit{minimally twisting} if every convex torus parallel to the boundary has slope $s$ between $s_0$ and $s_1$. Let $\alpha(s_i)$ denote the standard angle associated to $s_i$, thought of as sitting in $\mathbb{R}^2.$ The \textit{$I-$twisting} $\beta_I$ of $\xi$ is defined the total change in $\alpha$ as we traverse $T^2\times [0,1]$ in the $I-$direction. See Section 0.0.1 in \cite{hondatight2} for the precise definition. 

Now let $C$ be a $T^2$-bundle over $S^1$. A tight $\xi'$ on $C$ is called \textit{minimally twisting in the $S^1-$direction} if every splitting of $C$ along a convex torus isotopic to a fiber yields a minimally twisting $(T^2\times I,\xi)$. The \textit{$S^1-$twisting} $\beta_{S^1}$ of $\xi'$ is defined to be the supremum, over all convex tori $T$ isotopic to a fiber, of $l\pi$, where $l\in\mathbb{Z}^+$ and $l\pi \le \beta_I < (l+1)\pi$ on the $T^2\times I$ obtained by cutting along $T$.

Note that by definition, $(T^2\times I,\xi)$ has $I-$twisting $2n\pi$ or $(2n+1)\pi$ if and only if $\text{tor}(T^2\times I,\xi,[T^2\times\{pt\}])=n$. Similarly, $(C,\xi')$ has $S^1-$twisting $2n\pi$ or $(2n+1)\pi$ if and only if $\text{tor}(C,\xi',[T])=n$.


\section{Proof of Theorem \ref{thm:fillablestrs}}\label{fillable}

\subsection{Decomposing $Y_{\pm}$}\label{sfs}
Let $C_{\pm}$ denote a plumbed 3-manifold obtained as the boundary of a length $n>1$ cyclic plumbing, $Z_{\pm}$, as depicted in Figure \ref{generalcyclicplumbing}, where $a_i\ge2$ for all $i$. Then $C_{\pm}$ is a $T^2-$bundle over $S^1$. Endow $T^2\times[0,1]=\mathbb{R}^2/\mathbb{Z}^2\times[0,1]$ with the coordinates $(\textbf{x}^T,t)=(x,y,t)$. Then by Theorem 6.1 of \cite{neumann}, $C_{\pm}$ is of the form $T^2\times[0,1]/(\textbf{x},1)\sim(\pm B\textbf{x},0)$, where $$B=B(a_1,...,a_n)=\begin{pmatrix}
p&q\\
-p'&-q'\\
\end{pmatrix}, \frac{p}{q}=[a_1,...,a_n], \textrm{ and } \frac{p'}{q'}=[a_1,...,a_{n-1}].$$
Note that since det$B=1$, we have $p'q-q'p=1$.

Let $Y_{\pm}$ denote the plumbed 3-manifold obtained as the boundary of the plumbing, $X_{\pm}$, depicted in Figure \ref{generalsingularplumbing}, which has a cycle of length $n>1$ and where $a_i,z_j\ge 2$ for all $i,j$ and $a_1\ge3$. Let $T\subset Y$ be a torus associated with the plumbing operation that plumbs together the $-a_1$- and $-a_n$-framed vertices. Cutting along this torus, we obtain a manifold, $Y'_{\pm}$ with two torus boundary components, $T_0$ and $T_1$.  It is easy to see that $Y'_{\pm}$ is a Seifert fibered space over the annulus with a single singular fiber $F$, given by the arm with framings $(-z_1,...,-z_n)$. This structure can be built explicitly using the methods of \cite{orlikseifertmanifolds}.


$Y_{\pm}'$ can be obtained by starting with $T^2\times[0,1]$ and performing $-\frac{r}{s}=[-z_1,...,-z_m]$-surgery  along a curve $S^1\times\{pt\}\times\{pt\}\subset T^2\times(0,1)$ (See Figure \ref{gamma}). The framing is defined with respect to the $S^1-$direction, $\lambda$. The core of the solid torus obtained after surgery along this curve is the singular fiber $F$. Let $V$ be a tubular neighborhood of $F$. Then $Y_{\pm}' - V\cong S^1\times\Sigma$, where $\Sigma$ is a pair of pants (See Figure \ref{circlepants}). Identify $\partial V$ with $\mathbb{R}^2/\mathbb{Z}^2$ by choosing $(1,0)^T$ as the meridional direction and $(0,1)^T$ as the longitudinal direction and let $T_2$ denote the boundary component of $-\partial(S^1\times\Sigma)$ that is glued to $\partial V$.  Let $T_0=T^2\times\{0\}$ and $T_1=-T^2\times\{1\}$ and identify $T_i$ with $\mathbb{R}^2/\mathbb{Z}^2$ by setting $(1,0)^T=\mu$ as the direction given by $-\partial(\{pt\}\times\Sigma)$ and $(0,1)^T=\lambda$ as the direction given by the $S^1-$fiber. With this identification, the gluing map $T_1\to T_0$ is now given by  
$$A=\begin{pmatrix}
-p & q\\
p' & -q'\\
\end{pmatrix}$$ where $\frac{p}{q}=[a_1,...,a_n]$ and $\frac{p'}{q'}=[a_1,...,a_{n-1}]$. Moreover, the gluing map $g:\partial V \to -\partial(S^1\times\Sigma)$ is given by
$$g=\begin{pmatrix}
r&r'\\
-s&-s'\\
\end{pmatrix}
$$
where $\frac{r'}{s'}=[z_1,...,z_{m-1}]$. In particular, det$(g)=r's-s'r=1$. With these conventions set up, we have $-\partial(Y_{\pm}'-V)=T_0+T_1+T_2$.

\begin{figure}[h]
\centering
\begin{subfigure}{.45\textwidth}
\centering
\includegraphics[scale=.5]{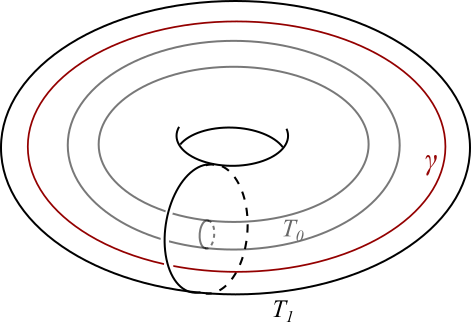}
\caption{The surgery curve $\gamma$ in $T^2\times[0,1]$}\label{gamma}
\end{subfigure}
\begin{subfigure}{.45\textwidth}
\centering
\includegraphics[scale=.5]{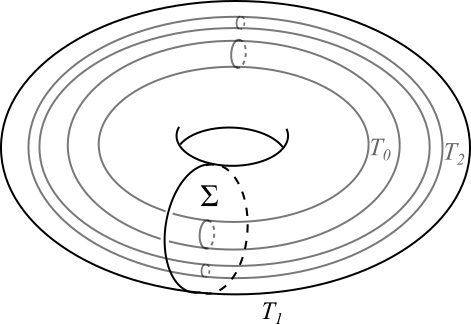}
\caption{$Y'_{\pm}-V\cong S^1\times\Sigma$}\label{circlepants}
\end{subfigure}
\caption{}\label{decomposition}
\end{figure}

We end this section by defining a notion of twisting analogous to the notions defined in Section \ref{twist}. Let $\xi$ be a tight contact structure on $Y_{\pm}$ and let $T\subset Y_{\pm}$ be the incompressible torus described above.

\begin{definition} The \textit{twisting} of $\xi$ is the the supremum, over all toric annuli $T^2\times I$ with $T^2\times\{pt\}$ isotopic to $T$, of $l\pi$, where $l\in\mathbb{Z}^+$ and $T^2\times I$ has $I-$twisting $l\pi$. We say $\xi$ is \textit{minimally twisting} if there exists no such toric annulus.\end{definition}

\begin{remark} Note that $\xi$ has twisting $2n\pi$ or $(2n+1)\pi$ if and only if $\text{tor}(Y_{\pm},\xi,[T])=n$.\end{remark}

\subsection{The upper bound}
Let $Y=Y_{\pm}$. We will distinguish between these two cases when necessary. Let $\xi$ be a tight contact structure on $Y$. Using the notation from Section \ref{sfs}, let $T$ be an incompressible convex torus that we can cut along to obtain $Y'$ and let $\Gamma_T$ denote the dividing set. After cutting along $T$, let $\Gamma_{T_i}$ denote the image of the dividing set on $T_i$ for $i=0,1$. With the coordinates described in Section \ref{sfs}, let $\Gamma_{T_0}=a\mu+b\lambda$, where $(a,b)=1$. Since $T_0$ and $T_1$ are identified by the map $\pm A$, the dividing set on $T_1$ is of the form $\Gamma_{T_1}=(bq+aq')\mu+(bp+ap')\lambda$ for $Y_+$ and $\Gamma_{T_1}=-(bq+aq')\mu-(bp+ap')\lambda$ for $Y_-$. Now isotope the singular fiber $F$ so that it is Legendrian and has very negative twisting number $-m<<0$, relative to a fixed framing. Then we may take $V$ to be a standard tubular neighborhood of $F$ with convex boundary so that the slope of the dividing set is $-\frac{1}{m}$ and $\#\Gamma_{\partial V}=2$ (See section 2.3.2 of \cite{ghigginischonenberger}). Thus the dividing set on $T_2\subset-\partial(S^1\times\Sigma)$ is of the form $\Gamma_{T_2}=(-mr+r')\mu-(-ms+s')\lambda$ and $\#\Gamma_{T_2}=2$. 

The slopes of these three dividing curves are as follows:
$$\displaystyle s(\Gamma_{T_0})=\frac{b}{a} \qquad s(\Gamma_{T_1})=\frac{bp+ap'}{bq+aq'} \qquad s(\Gamma_{T_2})=-\frac{ms-s'}{mr-r'}$$

Notice, for all relatively prime $a$ and $b$, $\frac{bp+ap'}{bq+aq'}$ is a reduced fraction, since $(\alpha q-\beta q')(bp+ap')+(\beta p'-\alpha p)(bq+aq')=1$, where $\alpha,\beta$ are integers such that $\alpha a+\beta b=1$. Furthermore, $-1\le s(\Gamma_{T_2})<0$. We can view $A^{-1}$ as a real-valued function that maps the slopes on $T_0$ to the slopes on $T_1$ given by $f(x)=\frac{xp+p'}{xq+q'}$. Since $f$ is a decreasing function of each interval of its domain, we have the relationship between slopes on $T_0$ and $T_1$ shown in Figure \ref{slopenumberline}.

\begin{figure}[h]
\includegraphics[scale=.45]{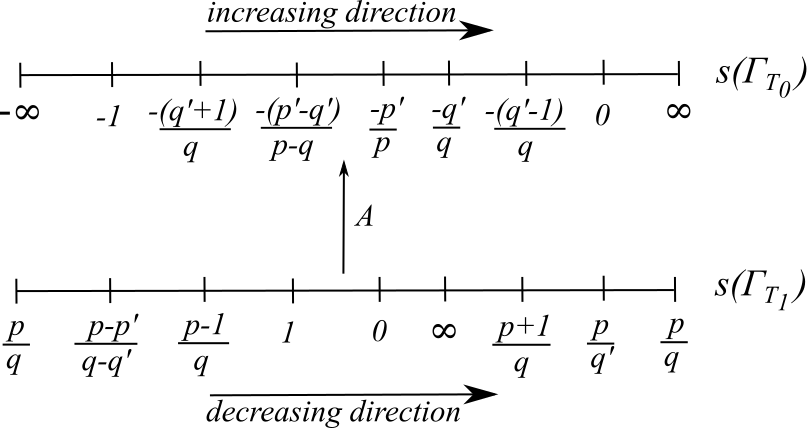}
\caption{Relationship between slopes on $T_0$ and $T_1$ via the gluing map}\label{slopenumberline}
\end{figure}

By the flexibility of Legendrian rulings (Theorem \ref{thm:flexibility}), we may arrange so that the Legendrian rulings on each torus has slope $\infty$ as long as the dividing sets do not have infinite slope. A convex annulus connecting two tori along such Legendrian rulings is called a \textit{vertical annulus}. Whenever possible, we will assume that the Legendrian rulings have infinite slope. Throughout this section, we assume that all tori and annuli are convex.

We have the following three cases:
\begin{itemize}
\item $|a|<|bq+aq'|$ if and only if $-\infty<s(\Gamma_{T_0})<-\frac{q'+1}{q}$ or $-\frac{q'-1}{q}<s(\Gamma_{T_0})<\infty$.
\item $|a|>|bq+aq'|$ if and only if $-\frac{q'+1}{q}<s(\Gamma_{T_0})<-\frac{q'-1}{q}$. 
\item $|a|=|bq+aq'|$ if and only if $s(\Gamma_{T_0})=-\frac{q'\pm1}{q}$. 
\end{itemize}

\noindent Let $2k$ be the number of dividing curves on $T_0$ and $T_1$. If $a\neq0$ and $bq+aq'\neq0$, then take a vertical annulus $A$ between $T_0$ and $T_1$. Then $|\Gamma_{T_0}\cdot\partial A|=|2ka|$ and $|\Gamma_{T_1}\cdot\partial A|=|2k(bq+aq')|$. By the Imbalance Principle (Theorem \ref{thm:imbalance}), if we are in the first case, then there exists a bypass along $T_1$. If we are in the second case, then there exists a bypass along $T_2$. If we are in the third case, then there are either bypasses along both tori or there are no bypasses. If $a=0$ (or $bq+aq'=0$), then we can take an annulus between a Legendrian divide of $T_0$ (or $T_1$, respectively) and a Legendrian ruling of $T_1$ (or $T_0$, respectively) and use the Imbalance Principle to see that there is a bypass along $T_0$ (or $T_1$, respectively). We will explore these cases in the following two propositions.

\begin{prop} If $a_i,z_j\ge2$ for all $i,j$ and $a_1\ge 3$, then we can choose $T$ so that $\#\Gamma_{T_0}=\#\Gamma_{T_1}=2$.\label{prop:divset2}\end{prop}
\begin{proof}

Suppose $a\neq0$, $bq+aq'\neq0$, and $\#\Gamma_{T_0}=\#\Gamma_{T_1}=2k$ for some $k>0$. Take a vertical annulus $A$ between $T_0$ and $T_1$. If $|a|\neq |bq+aq'|$, then by the Imbalance Principle (Theorem \ref{thm:imbalance}) there exists a bypass along a Legendrian divide of either $T_0$ or $T_1$ on $A$. Without loss of generality, assume $|bq+aq'|>|a|$. Then we may attach a bypass to $T_0$, giving us a new torus $T_0'$ isotopic to $T_0$ such that $s(\Gamma_{T_0'})=s(\Gamma_{T_0})$ and  $|\Gamma_{T_0'}\cdot\partial A|=|2(k-1)a|$. Thus, there exists an incompressible torus $T'$ isotopic to $T$ in $Y$ such that if we cut along $T'$ to obtain $Y'$, the new boundary tori $T_0'$ and $T_1'$ have the same boundary slopes as $T_0$ and $T_1$, but with two fewer dividing curves. Continuing this recutting process, we are able to arrange that $\#\Gamma_{T_0}=\#\Gamma_{T_1}=2$.

If $|a|=|bq+aq'|$, then $s(\Gamma_{T_0})=\frac{b}{a}=-\frac{q'\pm1}{q}$ and so $s(\Gamma_{T_1})=\frac{p\pm1}{q}$ (Note that these fractions may not be reduced, but their reduced fractions will still have the same denominators by Lemma \ref{lem2} in the Appendix). Assume the fractions are reduced. Take a vertical annulus between $T_0$ and $T_1$. If there exist bypasses along $T_0$ and $T_1$, we can attach the bypasses to lower $k$ and recut $Y$ along one of these new tori. If we can continue this until $k=1$, then we are done. Suppose there exists a $k>1$ such that there are no more bypasses. Then we may use the Edge Rounding Lemma (Lemma \ref{lem:edgeround}) to obtain a torus parallel to $-T_2$ with two dividing curves of slope $\frac{p-q'-2}{q}+\frac{1}{kq}>\frac{p-q'-2}{q}$ (if $\frac{b}{a}=-\frac{q'+1}{q}$) or $\frac{p-q'+2}{q}+\frac{1}{kq}>\frac{p-q'+2}{q}$ (if $\frac{b}{a}=-\frac{q'-1}{q}$). By Lemma \ref{lem1} in the Appendix, both of these slopes are greater than 1 since $a_1\ge3$. Thus, there is a torus, $T_2'$, parallel to $T_2$ with slope less than $-1$. Since $s(\Gamma_{T_2})>-1$, by Proposition \ref{prop:parallelslopes}, we can find a torus ``between"  $T_2'$ and $T_2$ with boundary slope $-1$ and two dividing curves. With abuse of notation, call this new torus $T_2$. Now, take a vertical annulus $A$ between $T_0$ and $T_2$. Then $|\Gamma_{T_2}\cdot\partial A|=2$ and $|\Gamma_{T_0}\cdot\partial A|=|2kq|$. Thus, by the Imbalance Principle, we may add bypasses to $T_0$ and lower $k$ until it is equal to 1. Recut $Y$ along this new torus to obtain the result. If $-\frac{q'\pm1}{q}$ is not reduced, then the same argument holds, since after edge rounding, we will obtain a torus of slope even greater than $\frac{p-q'-2}{q}+\frac{1}{kq}$ or $\frac{p-q'+2}{q}+\frac{1}{kq}$.

If $a=0$ so that $s(\Gamma_{T_0})=\infty$, then $s(\Gamma_{T_1})=\frac{p}{q}$. Take a vertical annulus $A$ from a Legendrian divide of $T_0$ to a Legendrian ruling of $T_1$. Then $|\Gamma_{T_0}\cdot\partial A|=0$ and $|\Gamma_{T_1}\cdot\partial A|=|2kq|$. We can thus add bypasses along $T_1$ until $k=1$. Recutting along this new torus, we obtain the result. We can similarly obtain the result if $bq+aq'=0$.\end{proof}

\begin{prop} If $a_i,z_j\ge2$ for all $i,j$ and $a_1\ge3$, then we can choose $T$ and $V$ so that $s(\Gamma_{T_0})=-1$, $s(\Gamma_{T_1})=\frac{p-p'}{q-q'}$, and $s(\Gamma_{T_2})=-1$.
\label{prop:boundaryslopes}\end{prop}

\begin{proof} 

First note that if we are able to arrange that either $s(\Gamma_{T_0})=-1$ or $s(\Gamma_{T_1})=\frac{p-p'}{q-q'}$, then we can easily obtain the result. Indeed, if we find a torus $T_0'$ parallel to $T_0$ with $s(\Gamma_{T_0'})=-1$, then we can recut $Y$ to obtain $s(\Gamma_{T_1})=\frac{p-p'}{q-q'}$ (or vice versa). We can then take a vertical annulus between $T_0$ and $T_2$ and, by the Imbalance Principle, add bypasses and use the Farey tessellation (Theorem \ref{thm:fareytess}) to decrease $s(\Gamma_{T_2})$ to $-1$.

First suppose $a=0$, so that $s(\Gamma_{T_0})=\infty$, then $s(\Gamma_{T_1})=\frac{p}{q}$. Take an annulus from a Legendrian divide of $T_0$ to a Legendrian ruling of $T_1$. Then we can add bypasses to $T_1$ to get a torus $T_1'$ with $s(\Gamma_{T_1'})=1$. Thus, by Proposition \ref{prop:parallelslopes}, there exists a torus between $T_1$ and $T_1'$ with slope $\frac{p-p'}{q-q'}$. We obtain a similar result if $s(\Gamma_{T_1})=\infty$. We now assume $a\neq0$ and $bq+aq'\neq0$.

Suppose $-1< s(\Gamma_{T_0})\le-\frac{q'+1}{q}$ and $\frac{p-1}{q}\le s(\Gamma_{T_1})<\frac{p-p'}{q-q'}$. Take a vertical annulus $A$ between $T_0$ and $T_2$. Suppose there exists a bypass on $A$ for either $T_0$ or $T_2$, or both. Then attach the bypasses, lowering the boundary slopes, and repeat the process. If we eventually reach $s(\Gamma_{T_0})=-1$, then we are done. Suppose we reach a step in which there are no more bypasses. Then since $-1<s(\Gamma_{T_0}), s(\Gamma_{T_2})<0$, we can use the Edge Rounding Lemma to find a torus $-T_1'$ parallel to $-T_1$ with boundary slope greater than $-2$. Thus $s(\Gamma_{T_1'})<2$. By Lemma \ref{lem1} in the Appendix, $s(\Gamma_{T_1})\ge\frac{p-1}{q}\ge2$. Thus, by Proposition \ref{prop:parallelslopes}, there exists another torus parallel to $T_1$ with slope $2$. With abuse of notation, call this new torus $T_1$. Now take a vertical annulus between $T_1$ and $T_0$ and use the Imbalance Principle to add bypasses to $T_0$ until $s(\Gamma_{T_0})=-1$.

Next suppose $-\frac{q'+1}{q}<s(\Gamma_{T_0})<-\frac{q'-1}{q}$ (and $s(\Gamma_{T_1})>\frac{p-1}{q}$ or $s(\Gamma_{T_1})<\frac{p+1}{q}$). Then $|a|>|bq+aq'|$ and so if we take a vertical annulus between $T_0$ and $T_1$, we can add a bypass to $T_0$, increasing its boundary slope using the Farey tessellation (Theorem \ref{thm:fareytess}), recut, and repeat. Since 0 and $-1$ share an edge in the Farey tessellation, we will eventually obtain $-1\le s(\Gamma_{T_0})\le-\frac{q'+1}{q}$, which is handled above.

Now suppose $-\infty<s(\Gamma_{T_0})<-1$ or $-\frac{q'-1}{q}<s(\Gamma_{T_0})<\infty$. Then $|a|<|bq+aq'|$. Taking a vertical annulus between $T_0$ and $T_1$, by the Imbalance Principle, we can add a bypass to $T_1$, recut, and repeat. Now, since $\frac{p-p'}{q-q'}<s(\Gamma_{T_1})<\frac{p+1}{q}$, by adding bypasses, recutting, and repeating, we eventually obtain $1\le s(\Gamma_{T_1})\le\frac{p-p'}{q-q'}$ (and $-1<s(\Gamma_{T_0})\le-\frac{p'-q'}{p-q}$), which is handled above.

Finally suppose $s(\Gamma_{T_0})=-\frac{q'-1}{q}$ and  $s(\Gamma_{T_1})=\frac{p+1}{q}$. Take a vertical annulus between $T_0$ and $T_1$. Then there either exists bypasses along both tori or along neither, since by Lemma \ref{lem2} in the Appendix, these slopes have the same denominator. If there do exist bypasses, we may add a bypass to $T_0$ to decrease its slope. Recut along this new torus to obtain the case $-\frac{q'+1}{q}<s(\Gamma_{T_0})<-\frac{q'-1}{q}$, which is handled above. If there do not exist bypasses, then as in the proof of Proposition \ref{prop:divset2}, we can use the Edge Rounding Lemma and Proposition \ref{prop:parallelslopes} to obtain $s(\Gamma_{T_2})=-1$. Now, take a vertical annulus between $T_0$ and $T_2$ and use the Imbalance Principle to add bypasses to $T_0$ until $s(\Gamma_{T_0})=-1$. \end{proof}

\begin{remark} In this proof we started with $s(\Gamma_{T_2})=-\frac{ms-s'}{mr-r'}$, for $m>>0$, and ended up with $s(\Gamma_{T_2})=-1$ after attaching bypasses. Thus, by Proposition \ref{prop:parallelslopes}, there is a convex torus $T_2'$ isotopic to $T_2$ with boundary slope $-\frac{s-s'}{r-r'}$. Equivalently, viewed from $V$, $s(\Gamma_{T_2'})=-1$ and $s(\Gamma_{T_2})=-\frac{r-s}{r'-s'}$. Thus, $V$ contains a toric annulus $T_2\times[1,2]$ such that $s(\Gamma_{T_2\times\{1\}})=-\frac{r-s}{r'-s'}$ and $s(\Gamma_{T_2\times\{2\}})=-1$. This fact will be used in the proof of Proposition \ref{prop2}. \label{remark1}\end{remark}

The following propositions consider \textit{basic slices} contained in $T^2\times I$ and related notions. See Section 4.3 in \cite{hondatight1} or Section 2.3 in \cite{ghigginischonenberger} for relevant definitions and results involving basic slices.

\begin{prop} If $\xi$ is not minimally twisting, then $Y_+$ has twisting $2l\pi$ for some $l\in\mathbb{Z}^+$. If $\xi$ is minimally twisting, then there are no vertical Legendrian curves with twisting number 0 in $Y_+'-V$.\label{prop:verticallegendrian}\end{prop}

\begin{proof}
By Proposition \ref{prop:boundaryslopes}, we may assume that $s(\Gamma_{T_0})=-1$, $s(\Gamma_{T_1})=\frac{p-p'}{q-q'}$, and $s(\Gamma_{T_2})=-1$. Suppose there is a vertical Legendrian curve, $\gamma$, with twisting number 0 in $Y_+'-V$. Take vertical annuli from $\gamma$ to $T_i$ for all $i$. Then $\gamma$ does not intersect $\Gamma_{A}$ and so we may use the Imbalance Principle to add bypasses to each torus until $s(T_i)=\infty$ for all $i$. 

 
There are three copies of $T^2\times I$ embedded in $Y_+'-V=S^1\times\Sigma$, namely $T_i\times I$, where $T_i\times\{0\}=T_i$ and  $T_i\times\{1\}$ has slope $\infty$ for all $i$. Since we obtained these by attaching bypasses, each $T_i\times I$ is minimally twisting, $T_0\times I$ and $T_2\times I$ each have a single basic slice, and $T_1\times I$ has $k\ge2$ basic slices, $T_1\times [0,\frac{1}{k}],...,T_1\times [\frac{k-1}{k},1]$. Since $s(\Gamma_{T_1})=\frac{p-p'}{q-q'}>\frac{p-1}{q}\ge2$ (by Lemma \ref{lem1} in the Appendix), $T_1\times\{\frac{k-3}{k}\}$ has boundary slope 2, $T_1\times \{\frac{k-2}{k}\}$ has boundary slope $1$, and $T_1\times \{\frac{k-1}{k}\}$ has boundary slope 0.

After possibly recutting, we may assume that there does not exist a torus $T_1'$ isotopic to $T_1$ such that the toric annulus between $T_1'$ and $T_1$ is non minimally twisting. Then, if possible, extend $T_0\times[0,1]$ to a non minimally twisting toric annulus $T_0\times[0,2]$ such that $T_0\times\{2\}$ has boundary slope $\infty$ and $T_0\times[1,2]$ has $I-$twisting $j\pi$, where $j\in\mathbb{Z}^+$. Assume that $T^2\times[0,2]$ is the toric annulus with largest $I-$twisting in $Y'_+-V$ with the prescribed boundary data. By Proposition 5.4 in \cite{hondatight1}, we may assume that $T_0\times\{\frac{3}{2}\}$ has boundary slope $-1$ and two dividing curves and that $T_0\times[\frac{3}{2},2]$ is a basic slice. If there does not exist such an extension, then we write $j=0$.




We now show that the signs of the basic slices $T_1\times [\frac{k-1}{k},1]$ and $T_0\times [\frac{3}{2},2]$ must be different (after choosing the sign convention to be given by selecting $(0,1)^T$ as the vector associated to $T_i\times\{1\}$ for $i=1,2$ and $T_0\times\{2\}$). Assume otherwise. By Honda's Gluing Theorem (Theorem 4.25 in \cite{hondatight1}), $T_1\times[\frac{k-2}{k},\frac{k-1}{k}]$ must have the same sign as  $T_1\times [\frac{k-1}{k},1]$. Now, since the basic slices of $T_1\times [\frac{k-2}{k},1]$ and $T_0\times I$ all have the same sign, by Lemma 4.13 in \cite{ghigginischonenberger}, there exists a vertical annulus between $T_1\times\{\frac{k-2}{k}\}$ and $T_0\times\{\frac{3}{2}\}$ that has no boundary-parallel dividing curves. Thus, by the Edge Rounding Lemma, we can obtain a torus, $T_2'$ parallel to $T_2\times\{1\}$ of slope $-1$. Thus, by Proposition \ref{prop:parallelslopes}, there must exist a torus $T_2''$ between $T_2\times\{1\}$ and $T_2'$ with boundary slope $-\frac{s'}{r'}$. On $\partial V$, the slope of this dividing set is $0$. But, any contact structure on $V\cong S^1\times D^2$ with boundary slope $0$ contains an overtwisted disk. Thus the signs of the basic slices $T_1\times [\frac{k-1}{k},1]$ and $T_0\times [\frac{3}{2},2]$ must be different. In the case of $j=0$, the same argument shows that the signs of $T_1\times [\frac{k-1}{k},1]$ and $T_0\times [0,1]$ are the same.

Suppose that $T_1\times [\frac{k-1}{k},1]$ has sign $\epsilon$ and $T_0\times [\frac{3}{2},2]$ has sign $-\epsilon$. Recut $Y$ along $T_1\times\{1\}$. Then we can thicken $T_0\times[0,2]$ to a toric annulus $T_0\times[-1,2]$, where $s(\Gamma_{T_0\times\{-1\}})=-\frac{q'}{q}$ and $T_0\times[-1,-\frac{k-1}{k}]$ is the image of $T_1\times[\frac{k-1}{k},1]$ under the recutting (i.e. the map $A$). Thus $T_0\times[-1,-\frac{k-1}{k}]$ is a basic slice and has sign $-\epsilon$. Since $T_0\times[-1,2]$ has $I-$twisting greater than $j\pi$ and less than $(j+1)\pi$, it admits exactly two tight contact structures (see Section 5.2 in \cite{hondatight1} for details). By the definitions of these contact structures, since the signs of $T_0\times[-1,\frac{k-1}{k}]$ and $T_0\times[1,2]$ are the same, $j$ must necessarily be odd. Note that this now excludes the case $j=0$.

Once again, recut $Y$ along $T_0\times\{0\}$ to return to the configuration we had before the previous paragraph. Assume $T_2\times [0,1]$ also has sign $\epsilon$. The case in which $T_2\times [0,1]$ has sign $-\epsilon$ is analogous. By Lemma 4.13 in \cite{ghigginischonenberger}, since the signs of the basic slices of $T_2\times[0,1]$ and $T_1\times[\frac{k-2}{k},1]$ are the same, there exists a vertical annulus between $T_1\times\{\frac{k-2}{k}\}$ and $T_2$ with no boundary parallel dividing curves. Thus, we can cut along this vertical annulus and use the Edge Rounding lemma to find a torus $T_0'$ parallel to $T_0\times\{2\}$ with boundary slope $-1$. Thus we have a toric annulus, $T_0\times[0,3]$ with $I-$twisting equal to $(j+1)\pi\in2\mathbb{Z}$.



Let $S^1\times\Sigma'\subset Y_+'-V=S^1\times\Sigma$ be defined by $S^1\times\Sigma=(S^1\times\Sigma')\cup(T_0\times[0,3])$. Then $S^1\times\Sigma'$ contains no vertical Legendrian curves with twisting number 0. Otherwise, we would be able to find an extension of $T_0\times[0,3]$ with $I-$twisting $(j+2)\pi$, contradicting the original choice of $j$. Thus, $Y$ contains even twisting.

Now, if $Y$ is minimally twisting, then there cannot be any vertical Legendrian curves with twisting number 0. Otherwise, following the arguments above, $Y$ would automatically have twisting at least $\pi$.
\end{proof}


The following result for $Y_-$ follows from arguments analogous to those in the proof of Proposition \ref{prop:verticallegendrian}.

\begin{cor} If $\xi$ is not minimally twisting, then $Y_-$ has twisting $(2l-1)\pi$ for some $l\in\mathbb{Z}^+$. If $\xi$ is minimally twisting, then $Y'_--V$ contains no vertical Legendrian curves with twisting number 0.\label{coretwisting}\end{cor}

\begin{remark} The previous proposition and corollary imply that $(Y_+,\xi)$ has no Giroux torsion if and only if it is minimally twisting and $(Y_-,\xi)$ has no Giroux torsion if and only if it is either minimally twisting or has twisting $\pi$.\label{GTremark} \end{remark}

\begin{prop} Let $a_i,z_j\ge2$ for all $i,j$ and $a_1\ge3$. Then $Y_+$ admits at most $(a_1-1)\cdot\cdot\cdot(a_n-1)(z_1-1)\cdot\cdot\cdot(z_m-1)$ minimally twisting tight contact structures. Moreover, $Y_+$ admits only even twisting tight contact structures and in particular for each $l\in\mathbb{Z}^+$, it admits at most $z_1(z_2-1)\cdot\cdot\cdot(z_m-1)$ tight contact structures with twisting $2l\pi$. By Remark \ref{GTremark}, $Y_+$ admits at most $(a_1-1)\cdot\cdot\cdot(a_n-1)(z_1-1)\cdot\cdot\cdot(z_m-1)$ tight contact structures with no Giroux torsion.\label{prop1}\end{prop}

\begin{proof}

For convenience, let $Y=Y_+$. By Proposition \ref{prop:boundaryslopes}, we may assume $s(\Gamma_{T_0})=-1$, $s(\Gamma_{T_1})=\frac{p-p'}{q-q'}$, and $s(\Gamma_{T_2})=-1$. First assume $\xi$ has is minimally twisting. Take a vertical annulus $A$ between $T_0$ and $T_2$. If there exists boundary parallel dividing curves on $A$, then we can add bypasses and obtain a torus parallel to $T_2$ with infinite slope, which contradicts Proposition \ref{prop:verticallegendrian}. Thus, there do not exist boundary parallel dividing curves. Cutting along $A$ and edge rounding, we obtain a torus $T_1'$ parallel to $T_1$ with boundary slope 1. Moreover, the toric annulus $T_1\times[0,1]$ between $T_1$ and $T_1'$ must have minimal $I-$twisting, by assumption.

Let $S^1\times\Sigma'\subset S^1\times\Sigma=Y'-V$ have boundary $-T_0 -T_1' -T_2$. Then $S^1\times\Sigma=(S^1\times\Sigma')\cup (T_1\times [0,1])$, where $T_1\times\{0\}=T_1$ and $T_1\times\{1\}=T_1'$. To find an upper bound on the number of minimally twisting tight contact structures on $Y'$, we need only find the number of such structures on the pieces $S^1\times\Sigma, T^2\times I,$ and $V$.

First, since $s(\Gamma_{T_2})=-1$, we have that $s(\Gamma_{\partial V})=-\frac{r-s}{r'-s'}=-[z_m,...,z_1-1]$. The proof of the latter equality can be found in the Appendix (Lemma \ref{lem4}). Thus by Theorem 2.3 in \cite{hondatight1}, $V$ admits $(z_1-1)\cdot\cdot\cdot(z_m-1)$ tight contact structures. Changing the coordinates on $T_1\times[0,1]$ by reversing the orientation on $T^2$, we obtain $s(\Gamma_{T_1})=-\frac{p-p'}{q-q'}=-[a_1,...,a_n-1]$ and $s(\Gamma_{T_1'})=-1$. The proof of the former equality can also be found in the Appendix (Lemma \ref{lem3}). By Theorem 2.2 in \cite{hondatight1}, there are $(a_1-1)\cdot\cdot\cdot(a_n-1)$ minimally twisting tight contact structures on $T_1\times [0,1]$. Since $S^1\times\Sigma'$ has no vertical Legendrian curves of twisting number 0, by Lemma 5.1-4c in \cite{hondatight2}, $S^1\times\Sigma$ admits $2-(1-1+1)=1$ tight contact structures. 

Therefore, $Y'$ admits at most $1\cdot(a_1-1)\cdot\cdot\cdot(a_n-1)(z_1-1)\cdot\cdot\cdot(z_m-1)$ minimally twisting tight contact structures. Gluing the ends of $Y'$ together via $A$ to obtain $Y$, we have that $Y$ also admits at most $(a_1-1)\cdot\cdot\cdot(a_n-1)(z_1-1)\cdot\cdot\cdot(z_m-1)$ minimally twisting tight contact structures. 

Now suppose $(Y,\xi)$ is not minimally twisting. Then by Proposition \ref{prop:verticallegendrian}, $(Y,\xi)$ must have twisting $2l\pi$ for some $l\in\mathbb{Z}^+$. Decompose $Y$ as above so that $Y'-V=S^1\times\Sigma=(S^1\times\Sigma')\cup (T_1\times [0,1])$, where $T_1\times\{0\}=T_1$ and $T_1\times\{1\}=T_1'$. By the proof of Proposition \ref{prop:verticallegendrian}, we may assume that there exists a toric annulus $T_0\times[0,3]\subset S^1\times\Sigma'$ with $I-$twisting $2l\pi$, where $T_0\times\{0\}=T_0$, $s(\Gamma_{T_0\times\{1\}})=\infty$, $s(\Gamma_{T_0\times\{\frac{3}{2}\}})=-1$, $s(\Gamma_{T_0\times\{2\}})=\infty$, and $s(\Gamma_{T_0\times\{3\}})=-1$. Let $S^1\times \Sigma''\subset S^1\times \Sigma'$ be such that $S^1\times\Sigma=(S^1\times\Sigma'')\cup (T_0\times[0,3])\cup(T_1\times[0,1])$. Recut $Y$ along $T_1'=T_1\times\{1\}$ and thicken $T_0\times[0,3]$ to $T_0\times[-1,3]$, where $s(\Gamma_{T_0\times\{-1\}})=-\frac{p'-q'}{p-q}$. Then $Y'-V=(S^1\times\Sigma'')\cup (T_0\times[-1,3])$ and  $\partial (S^1\times \Sigma'')=-(T_0\times\{3\}) -T_1' -T_2$. As mentioned in the proof of Proposition \ref{prop:verticallegendrian}, $S^1\times\Sigma''$ does not contain any vertical Legendrian curves with twisting number 0 and $T_0\times[-1,3]$ admits exactly two tight contact structures. By Lemma 5.1-4c in \cite{hondatight2}, $S^1\times \Sigma''$ admits exactly one tight contact structure and by Theorem 2.3 in \cite{hondatight1}, $V$ admits $(z_1-1)\cdot\cdot\cdot(z_m-1)$ tight contact structures. Thus, for each $l\in\mathbb{Z}^+$, $Y$ admits at most $2(z_1-1)\cdot\cdot\cdot(z_m-1)$ tight contact structures with twisting $2l\pi$. If $z_1=2$, then this number is the same as $z_1(z_2-1)\cdot\cdot\cdot(z_m-1)$ and we are done. Assume $z_1>2$. 

We claim that $2(z_1-2)(z_2-1)\cdot\cdot\cdot(z_m-1)$ of these tight contact structures on $Y$ pair off isotopically, yielding a total of $2(z_1-1)\cdot\cdot\cdot(z_m-1)-(z_1-2)(z_2-1)\cdot\cdot\cdot(z_m-1)=z_1(z_2-1)\cdot\cdot\cdot(z_m-1)$ tight contact structures on $Y$. Since $s(\Gamma_{T_2})=-1$, we have $s(\Gamma_{\partial V})=-\frac{r-s}{r'-s'}=-[z_m,...,z_2,z_1-1]$. Let $T_2'\subset S^1\times\Sigma$ be a torus isotopic to $T_2$ such that $s(\Gamma_{T_2'})=\infty$ (which exists by the proof of Proposition \ref{prop:verticallegendrian}) and let $V'$ be the corresponding thickening of $V$. Then $s(\Gamma_{\partial V'})=-\frac{r}{r'}=-\frac{r}{\overline{s}}=-[z_m,...,z_1]$, where $\overline{s}$ is the unique integer such that $1<\overline{s}<r$ and $s\overline{s}\equiv 1\text{mod}r$. By Theorem 2.3 in \cite{hondatight1}, $V$ and $V'$ admit $(z_1-1)\cdot\cdot\cdot(z_m-1)$ and $z_1(z_2-1)\cdot\cdot\cdot(z_m-1)$ tight contact structures, respectively. 

Let $B=T_2\times [0,1]\subset V'$ be the basic slice with $T_2\times\{0\}=\partial V'$ and $T_2\times\{1\}=\partial V$ (i.e. $B=V'-V$). Then $B$ admits 2 tight contact structures, which depend on the sign of the relative Euler class. We claim that exactly $(z_1-2)(z_2-1)\cdot\cdot\cdot(z_m-1)$ of the $z_1(z_2-1)\cdot\cdot\cdot(z_m-1)$ tight contact structures of $V'$ have the property that the sign of the basic slice $B$ can be either positive or negative after shuffling (Lemma 4.14 of \cite{hondatight1}). Thicken $B$ to the toric annulus $T_2\times[0,2]=B\cup T_2\times[1,2]\subset V'$, which has boundary slopes $s(\Gamma_{T_2\times\{0\}})=-\frac{r}{r'}$, $s(\Gamma_{T_2\times\{1\}})=-\frac{r-s}{r'-s'}$, and $s(\Gamma_{T_2\times\{2\}})=-1$ (this thickening exists by Remark \ref{remark1}). Consider the first \textit{continued fraction block} (see Section 4.4.5 in \cite{hondatight1}) of $T_2\times[0,2]$, which is a toric annulus $T_2\times[0,\frac{3}{2}]$ satisfying $s(\Gamma_{T_2\times\{\frac{3}{2}\}})=-\frac{r-(z_1-1)s}{r'-(z_1-1)s'}$. This block admits $z_1$ tight contact structures, of which only two do not have the desired property, namely the two contact structures whose basic slices all have the same sign. Thus $T_2\times[0,\frac{3}{2}]$ admits $z_1-2$ contact structures that satisfy the desired property. Moreover, $T_2\times[\frac{3}{2},2]$ admits $(z_2-1)\cdot\cdot\cdot(z_m-1)$ nonisotopic tight contact structures which remain nonisotopic in $T_2\times[0,2]$. Thus, there are $(z_1-2)(z_2-1)\cdot\cdot\cdot(z_m-1)$ tight contact structures on $T_2\times[0,2]$ (and thus on $V'$) such that the sign of the basic slice $B$ can be either positive or negative after shuffling.

\begin{figure}[h]
\begin{subfigure}{.4\textwidth}
\centering
\includegraphics[scale=.6]{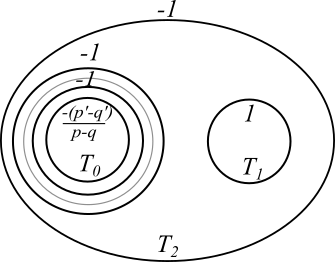}
\caption{}\label{pairofpants2}
\end{subfigure}
\begin{subfigure}{.5\textwidth}
\centering
\includegraphics[scale=.6]{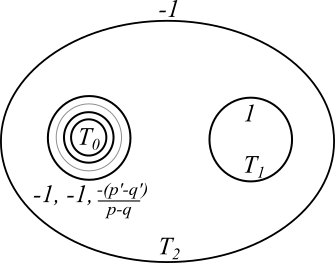}
\caption{}\label{pairofpants3}
\end{subfigure}
\caption{Consider a section of $S^1\times\Sigma$, which is a pair of pants. The numbers represent the slopes of the dividing curves of the corresponding tori. The gray circles represent tori with boundary slope $\infty$. In (B), the slopes below $T_0$ correspond to the black circles from outer- to inner-most.}
\end{figure}

Fix one such contact structure $\xi$ on $Y$. Decompose $Y$ as in the proof of Proposition \ref{prop:verticallegendrian}. Using the same notation as in that proof, recall that the signs of $T_0\times[\frac{3}{2},2]$ and $T_1\times[\frac{k-2}{k},1]$ are different. Choose the sign of $B$ to be the same as the sign of $T_1\times[\frac{k-2}{k},1]$. Following the proof of Proposition \ref{prop:verticallegendrian}, after recutting along $T_1\times\{\frac{k-2}{k}\}$ and appropriately edge rounding, there is a toric annulus $T_0\times[-1,3]$ with $I-$twisting greater than $2l\pi$ and less than $(2l+1)\pi$ such that $s(\Gamma_{T_0\times\{-1\}})=-\frac{p'-q'}{p-q}$ and $s(\Gamma_{T_0\times\{3\}})=-1$. See Figure \ref{pairofpants2}. Thus $\xi|_{T_0\times[-1,3]}$ is one of two possible tight contact structures. 

By edge rounding and recutting, we will be able to isotope $\xi$ so that $\xi|_{T_0\times[-1,3]}$ becomes the other contact structure, which we denote by $\xi'|_{T_0\times[-1,3]}$. To do this, ignore the basic slice $T_0\times[2,3]$, take vertical annuli from $T_0\times\{2\}$ to $T_1$ and from $T_0\times\{2\}$ to $T_2$, and add bypasses to $T_1$ and $T_2$ to obtain toric annuli $T_i\times[0,1]$ such that $T_i\times\{0\}=T_i$ and $s(\Gamma_{T_i\times\{1\}})=\infty$ for $i=1,2$. Notice that $T_2\times[0,1]=B$. Now, by shuffling, choose the opposite sign for the basic slice $B$ than we chose previously. Since the signs of the basic slices of $T_0\times[\frac{3}{2},2]$ and $B$ are the same, by applying Lemma 4.13 in \cite{ghigginischonenberger}, we can edge round (c.f. the proof of Proposition \ref{prop:verticallegendrian}) and thicken the toric annulus $T_1\times[\frac{k-2}{k},1]$ to $T_1\times[\frac{k-2}{k},2]$, where $s(\Gamma_{T_1\times\{2\}})=1$ and $T_1\times[\frac{k-2}{k},2]$ has $I-$twisting equal to $\pi$. As previously, recut $Y$ along $T_1\times\{2\}$, relabel the toric annulus as $T_0\times[-3,-1]$, and glue it to $T_0\times[-1,\frac{3}{2}]$. See Figure \ref{pairofpants3}. Note that $T_0\times[-1,\frac{3}{2}]$ and $T_0\times[-1,3]$ (as constructed in the previous paragraph) have the same boundary data. Moreover, based on the definitions of the two contact structures in question (see Section 5.2 in \cite{hondatight1}), it is easy to see that $\xi|_{T_0\times[-3,\frac{3}{2}]}=\xi'|_{T_0\times[-1,3]}$.\end{proof}

Adapting the proof of Proposition \ref{prop1} to the case of $Y_-$ and, in particular, invoking Corollary \ref{coretwisting}, we obtain the following analogous result. 

\begin{cor} Let $a_i,z_j\ge2$ for all $i,j$ and $a_1\ge3$. Then $Y_-$ admits at most $(a_1-1)\cdot\cdot\cdot(a_n-1)(z_1-1)\cdot\cdot\cdot(z_m-1)$ minimally twisting tight contact structures. Moreover, $Y_-$ admits only odd twisting tight contact structures and in particular for each $l\in\mathbb{Z}^+$, it admits at most $z_1(z_2-1)\cdot\cdot\cdot(z_m-1)$ tight contact structures with twisting $(2l-1)\pi$. By Remark \ref{GTremark}, $Y_-$ admits at most $(a_1-1)\cdot\cdot\cdot(a_n-1)(z_1-1)\cdot\cdot\cdot(z_m-1)+z_1(z_2-1)\cdots(z_m-1)$ tight contact structures with no Giroux torsion.\label{prop2}\end{cor}

\subsection{The lower bound}

\begin{prop} If $a_i,z_j\ge2$ for all $i,j$ and $a_1\ge3$, then $Y_+$ admits exactly $(a_1-1)\cdot\cdot\cdot(a_n-1)(z_1-1)\cdot\cdot\cdot(z_m-1)$ tight contact structures with no Giroux torsion, all of which are all Stein fillable.   \label{prop3}\end{prop}

\begin{proof} We can easily construct $(a_1-1)\cdot\cdot\cdot(a_n-1)(z_1-1)\cdot\cdot\cdot(z_m-1)$ Stein fillable contact structures for $Y_+$, by drawing suitable Kirby diagrams. Start with the obvious Kirby diagram of the plumbing $X_+$ and make every unknot Legendrian with $tb=-1$ and $r=0$, as in Figure \ref{legendrianhandlebodyX1}. Then to ensure we obtain a Stein structure, we must stabilize each $a_i-$framed unknot $a_i-2$ times and each $z_i-$framed unknot $z_i-2$ times. There are $a_i-1$ (resp. $z_i-1$) ways to stabilize each unknot. Thus, there are a total of $(a_1-1)\cdot\cdot\cdot(a_n-1)(z_1-1)\cdot\cdot\cdot(z_m-1)$ ways to stabilize the entire diagram. Since different kinds of stabilizations yield different rotation numbers, the resulting Stein structures have different first Chern classes and so the induced contact structures on the boundary are pairwise nonisotopic (by Theorem 1.2 in \cite{liscamatictightstrs}). Thus there are at least $(a_1-1)\cdot\cdot\cdot(a_n-1)(z_1-1)\cdot\cdot\cdot(z_m-1)$ Stein fillable contact structures on $Y_+$. By Corollary 2.2 of \cite{gayfillabletightstrs}, these contact structures have no Giroux torsion and are thus minimally twisting by Remark \ref{GTremark}. Coupling this with Proposition \ref{prop1}, we obtain the result.\end{proof}

\begin{figure}[h]
\begin{subfigure}{.5\textwidth}
\centering
\includegraphics[scale=.55]{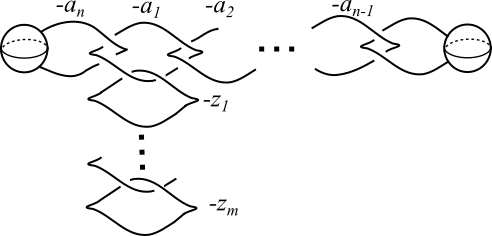}
\caption{Handle body diagram of $X_+$}\label{legendrianhandlebodyX1}
\end{subfigure}
\begin{subfigure}{.48\textwidth}
\centering
\includegraphics[scale=.55]{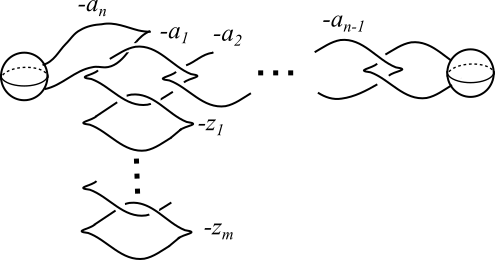}
\caption{Handle body diagram of $X_-$}\label{legendrianhandlebodyX2}
\end{subfigure}
\caption{Handle body diagrams of $X_{\pm}$ with smooth framings}\label{legendrianhandlebodyX}

\end{figure}

\begin{prop} For each $l\in\mathbb{Z}^+$, $Y_+$ admits exactly $z_1(z_2-1)\cdot\cdot\cdot(z_m-1)$ weakly fillable contact structures with twisting $2l\pi$.\label{prop3a} \end{prop}

\begin{proof} 
Let $C=C_+$ denote the $T^2$-bundle over $S^1$ obtained as the boundary of the cyclic plumbing depicted in Figure \ref{generalcyclicplumbing}. Equip $C$ with the unique universally tight contact structure $\xi_l$ with $S^1-$twisting $2l\pi$, where $l\in\mathbb{Z}^+$. Recall, using the notation of Section \ref{sfs}, that $C\cong T^2\times[0,1]/(x,1)$$\sim$$(-Bx,0)$, and $Y=Y_+$ is obtained from $C$ by performing surgery along $\gamma=S^1\times\{\text{pt}\}\times\{\text{pt}\}$, as depicted in Figure \ref{surgerycurve} (using the conventions set up in Section \ref{steincobordism}). To perform this surgery, we remove a solid torus neighborhood of $\gamma$ and glue in solid torus $V$ via the map $g=\begin{pmatrix}
r&r'\\
-s&-s'\\
\end{pmatrix}
$,
where $\frac{r}{s}=[z_1,...,z_m]$, $\frac{r'}{s'}=[z_1,...,z_{m-1}]$, and $-\partial(Y-V)$ is identified with $\mathbb{R}^2/\mathbb{Z}^2$ as in Section \ref{sfs}.

\begin{figure}[h]
\centering
\begin{subfigure}{.45\textwidth}
\centering
\includegraphics[scale=.5]{surgerycurve}
\caption{The surgery curve $\gamma$ in $C$}\label{surgerycurve}
\end{subfigure}
\begin{subfigure}{.45\textwidth}
\centering
\includegraphics[scale=.5]{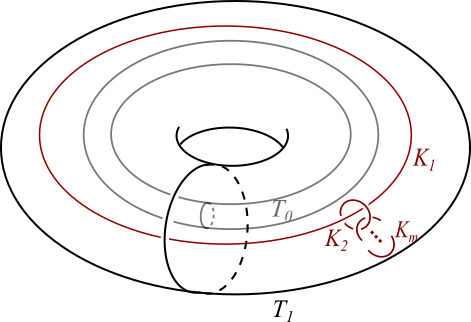}
\caption{The surgery curves $K_1,...,K_m$ after slam dunk}\label{surgerycurve2}
\end{subfigure}
\caption{}\end{figure}

It is known (see, for example, \cite{hondatight1}) that $\xi_l$ is the kernel of the 1-form $\alpha_l=\sin(\phi(t))dx+\cos(\phi(t))dy$, where $\phi'(t)>0$ and $2l\pi\le\sup_{t\in\mathbb{R}/\mathbb{Z}} \phi(t+1)-\phi(t)<(2l+1)\pi$. Moreover, since $Y$ is a hyperbolic torus bundle, the first inequality must also be strict. Thus there exists a toric annulus $T^2\times[0,1]$ with $I-$twisting $\pi$ embedded in $C$ with contact form $\alpha=\sin(\phi(t))dx+\cos(\phi(t))dy$, where $\phi(0)=-\frac{\pi}{2}$ and $\phi(1)=\frac{\pi}{2}$. 

Let $t_0\in(0,1)$ be as in Section \ref{steincobordism} and isotope $\gamma$ so that $\gamma=S^1\times\{pt\}\times\{t_0\}$ is a longitude on the torus $T_{t_0}$. Then it is clearly Legendrian and has twisting number 0 (with respect to the 0-framing). Perform consecutive (reverse) slam dunk moves in $T^2\times(0,1)$ starting with $\gamma$ to obtain a link $K_1\sqcup\cdot\cdot\cdot\sqcup K_m$ in $C$ (see Figure \ref{surgerycurve2}), where $K_1=\gamma$ and $K_i$ has framing $-z_i$ for all $i$. Then $Y$ is obtained by performing $-z_i$-surgery on $K_i$ for all $i$. Take the front projection of $K_1\sqcup\cdot\cdot\cdot\sqcup K_m$ satisfying $tb(K_1)=0$, $tb(K_i)=-1$ for $i\ge2$, and $r(K_i)=0$ for all $i$ (see Figure \ref{frontprojection}). Stabilize $K_1$ $z_1-$times and stabilize $K_i$ $(z_i-1)-$times for $i\ge2$. By Lemma \ref{chernlemma}, by attaching Stein $2-$handles along each $K_i$, we obtain $z_1(z_2-1)\cdot\cdot\cdot(z_m-1)$ distinct Stein cobordisms $(W,J_i)$ from $(C,\xi_l)$ to contact manifolds $(Y,\nu_i)$, where $1\le i\le z_1(z_2-1)\cdot\cdot\cdot(z_m-1)$. In \cite{dinggeiges}, it is shown that $(C,\xi_l)$ is weakly symplectically fillable. Thus by Corollary \ref{cobordismcor}, the contact structures $\nu_i$ are all tight and pairwise nonisotopic. Moreover, by construction, it is clear that $(Y,\nu_i)$ is weakly fillable and has twisting $2l\pi$ for all $i$. 
\end{proof}

\begin{figure}[h]
\centering
\includegraphics[scale=.55]{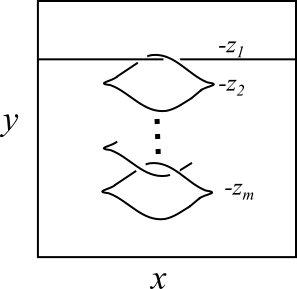}
\caption{The front projection of $K_1\sqcup\cdot\cdot\cdot\sqcup K_m$ with smooth framings}\label{frontprojection}
\end{figure}


\begin{prop} Let $a_i,z_j\ge2$ for all $i,j$ and $a_1\ge3$. $Y_-$ admits exactly $(a_1-1)\cdot\cdot\cdot(a_n-1)(z_1-1)\cdot\cdot\cdot(z_m-1)$ minimally twisting tight contact structures, all of which are Stein fillable. For each $l\in\mathbb{Z}^+$, $Y_-$ admits $z_1(z_2-1)\cdot\cdot\cdot(z_m-1)$ weakly fillable contact structures with twisting $(2l-1)\pi$. If $(a_1,...,a_n)$ is embeddable (as defined in \cite{gollalisca}), then the $z_1(z_2-1)\cdot\cdot\cdot(z_m-1)$ contact structures with twisting $\pi$ are also Stein fillable.\label{prop4}\end{prop}

\begin{proof} As in the proof of Proposition \ref{prop3}, we can easily exhibit $(a_1-1)\cdot\cdot\cdot(a_n-1)(z_1-1)\cdot\cdot\cdot(z_m-1)$ pairwise nonisotopic contact structures as the boundaries of the Stein domains obtained by stabilizing the obvious handle body diagram depicted in Figure \ref{legendrianhandlebodyX2}. We now argue that these contact structures are minimally twisting. We know by \cite{gayfillabletightstrs}, that they have no Giroux torsion, but there is still the possibility that some of these contact structures have twisting $\pi$. Let $C=C_-$ denote the $T^2$-bundle over $S^1$ obtained as the boundary of the cyclic plumbing depicted in Figure \ref{generalcyclicplumbing}. Equip $C$ with a contact structure $\xi$ that is induced by a Stein structure on the plumbing. Such a Stein domain has a handle description consisting of the $1-$handle and the horizontal chain of unknots (with additional stabilizations) depicted in Figure \ref{legendrianhandlebodyX2}. By performing Legendrian surgery along the vertical chain of unknots (with additional stabilizations) in Figure \ref{legendrianhandlebodyX2}, we obtain $Y$ along with one of the contact structures in question. Moreover, all of the contact structures in question can be obtained this way. By Theorem 3.1 in \cite{gollalisca}, the induced contact structure on $C$ is virtually overtwisted and by \cite{hondatight2}, such contact $T^2$-bundles are minimally twisting. Thus the contact structures in question must also be minimally twisting.

Now, proceeding as in the proof of Proposition \ref{prop3a}, we obtain $z_1(z_2-1)\cdot\cdot\cdot(z_m-1)$ weakly fillable contact structures with twisting $(2l-1)\pi$ for all $l\in\mathbb{Z}^+$. Using the notation in the proof of Proposition \ref{prop3a} suitably adapted to our current context, if $(a_1,...,a_n)$ is embeddable, then by \cite{gollalisca}, $(C,\xi_1)$ is Stein fillable. Thus, by gluing the Stein filling together with the various Stein cobordisms $(W_i,J_i)$ from $(C,\xi_1)$ to $\{(Y,\nu_i)\}$, we see that the tight contact structures on $Y$ with twisting $\pi$ are Stein fillable. \end{proof}

\section{Some explicit Stein fillings of $Y_-$}\label{Y_-filling}

In this section, we will give a general description of the Stein fillings of $(Y_-,\nu_i)$, where the $\nu_i$ are the $z_1(z_2-1)\cdot\cdot\cdot(z_m-1)$ Stein fillable contact structures with twisting $\pi$ described at the end of the proof of Proposition \ref{prop4}. This description is similar to the description of the symplectic fillings of the canonical contact structure on lens spaces described by Lisca in \cite{lisca}. We first give smooth descriptions of the Stein fillings of hyperbolic $T^2$-bundles over $S^1$ found in \cite{gollalisca}. Let $C=C_-$ be the boundary of the negative cyclic plumbing with framings $(-a_1,...,-a_n)$ shown in Figure \ref{generalcyclicplumbing} such that $(a_1,...,a_n)$ is embeddable (as defined in \cite{gollalisca}). Consider the obvious surgery description of $C$. Then by blowing up with $1$-framed unknots and continually blowing down any resulting $-1$-framed unknots, we can obtain a surgery description of the dual graph (c.f. \cite{neumann}) with framings $(d_1,...,d_k)$. Denote the unknot with framing $d_i$ by $K_i$. Since $(a_1,...,a_n)$ is embeddable, there exists a blowup $(c_1,...,c_k)$ of $(0,0)$ such that $c_i\le d_i$ for all $i$. If we blow down $K_i$ $(d_i-c_i)$-times, we obtain the surgery description of $C$ in Figure \ref{fillingdescription}.

\begin{figure}[h!]
\centering
\includegraphics[scale=.5]{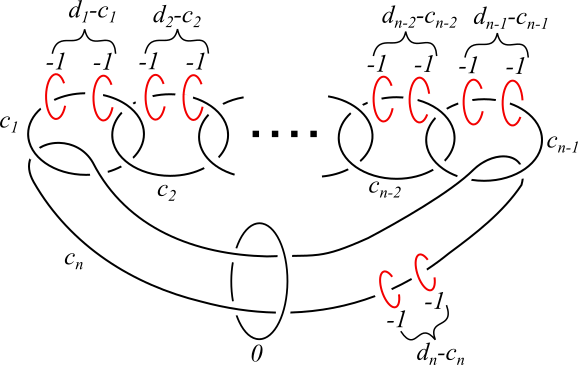}
\caption{A surgery diagram of $C$}\label{fillingdescription}
\end{figure}

To obtain a handle body diagram of the smooth filling of $C$, blow down the sequence $(c_1,...,c_k)$ appropriately until the chain becomes two $0-$framed unknots. Notice, we will be left with the $0$-framed Borromean rings along with the image of the $-1$-framed red curves, which are now complicated knots with various negative framings. Finally, change the two $0-$framed unknots resulting from blowing down the sequence $(c_1,...,c_k)$ to dotted circles. Then the resulting 4-manifold, $D$, is bounded by $C$. In \cite{gollalisca}, $D$ is given a Stein structure that induces the universally tight contact structure $\xi_1$ on $C$.

Similarly, let $Y=Y_-$ be the boundary of the negative cyclic plumbing shown in Figure \ref{generalsingularplumbing}. Then, consider its obvious surgery diagram and follow the above steps for the cycle portion. We will then obtain Figure \ref{fillingdescription} along with a chain of unknots with framings $(-z_2,...,-z_m)$ dangling from the image of the $-z_1$-framed unknot, which links $K_1\sqcup\cdot\cdot\cdot \sqcup K_n$ in a complicated way. Once again, to obtain the smooth filling of $Y$, on which we can place $z_1(z_2-1)\cdot\cdot\cdot(z_m-1)$ Stein structures, blow down the sequence $(c_1,...,c_k)$ appropriately until the chain becomes two $0-$framed unknots and then change those unknots to dotted circles. To see the various Stein structures, one would need to arrange the diagram appropriately. Since the induced contact structures are obtained via Legendrian surgery on $(C,\xi_1)$, by the remarks in the proof of Proposition \ref{prop4}, these contact structures are indeed the $\nu_i$. 

Drawing these diagrams is intractable in general, but in easy situations, it is manageable. For example, consider the cyclic plumbing $C$ in Figure \ref{cycle1}. Consider the obvious surgery diagram of the boundary. By blowing up with two $+1$-unknots located on either side of the leftmost $-3$-unknot and then blowing down consecutive $-1$-framed unknots, we obtain the Borromean rings with framings $0, k+3,$ and $j+3$. Next, blow up the unknots with framings $k+3$ and $j+3$, $(k+3)-$times and $(j+3)-$times, respectively. Finally turn the resulting two $0-$framed unknots into dotted circle notation. By isotoping appropriately, we obtain the handle body diagram depicted in Figure \ref{filling1}. By the arguments in \cite{gollalisca}, $C$ admits a Stein structure that induces the universally tight contact structure $\xi_1$. By computing the $d_3$-invariants of the three possible contact structures on $C$, it is easy to see that the Stein diagram we have drawn in Figure \ref{filling1} induces $\xi_1$. Similarly, if we begin with the plumbing $Y_-$ in Figure \ref{cycle2} and apply the same moves, we obtain the handle body diagram depicted in Figure \ref{filling2}. After stabilizing appropriately, we obtain $z_1(z_2-1)\cdot\cdot\cdot(z_m-1)$ nonisotopic tight contact structures on $Y_-$. Since these contact structures are obtained by Legendrian surgery on $(C,\xi_1)$, thus these contact structures are the $\nu_i$.

\begin{figure}[h]
\centering
\begin{subfigure}{.45\textwidth}
\centering
\includegraphics[scale=.55]{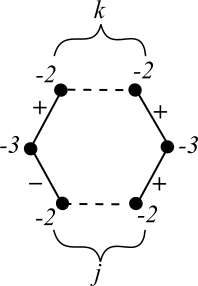}
\caption{The boundary of this plumbing is $C$.}\label{cycle1}
\end{subfigure}
\begin{subfigure}{.45\textwidth}
\centering
\includegraphics[scale=.55]{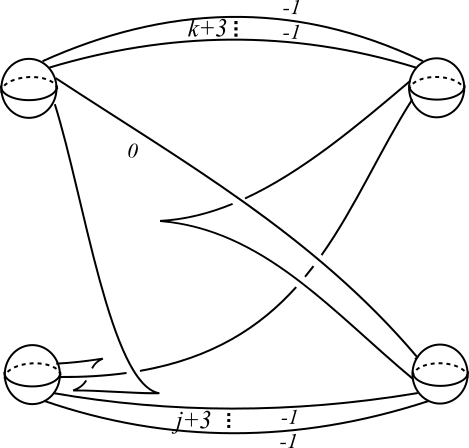}
\caption{Stein handlebody diagram of the filling of $(C,\xi_1)$. The framings are smooth framings.}\label{filling1}
\end{subfigure}
\begin{subfigure}{.45\textwidth}
\centering
\includegraphics[scale=.55]{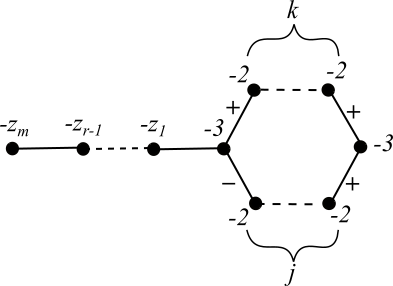}
\caption{The boundary of this plumbing is $Y_-$.}\label{cycle2}
\end{subfigure}
\begin{subfigure}{.45\textwidth}
\centering
\includegraphics[scale=.55]{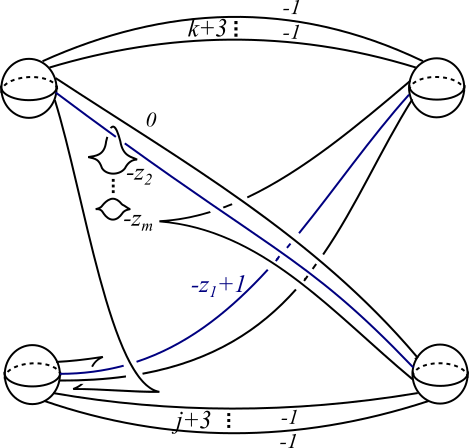}
\caption{Stein handlebody diagrams (without all stabilizations) of the fillings of $(Y,\nu_i)$ for all $i$. The framings are smooth framings.}\label{filling2}
\end{subfigure}
\caption{Stein fillings with smooth framings}
\end{figure}

\clearpage

\section{Appendix}\label{appendix}

Here we prove the some minor facts about continued fractions that are used throughout Section \ref{fillable}. Let $\frac{p}{q}=[a_1,...,a_n]$ and $\frac{p'}{q'}=[a_1,...,a_{n-1}]$, where $a_i\ge 2$ for all $i$.

\begin{lem} If $a_1\ge 3$, then $p\ge 2q+1>q+q'+1$.\label{lem1}\end{lem}

\begin{proof}  Let $t$ be unique the integer satisfying $\frac{q}{t}=[a_2,...,a_n]$. Note that $t,q'<q$. Thus $p=a_1q-t=(a_1-2)q+q+(q-t)\ge2q+1>q+q'+1$.\end{proof}

\begin{lem} $\frac{q'\pm1}{q}$ is a reduced fraction if and only if $\frac{p\mp1}{q}$ is a reduced fraction. Moreover, we have that $(q'\pm1, q)=(p\mp1,q)$. \label{lem2}\end{lem}

\begin{proof} Since $p'q-q'p=1$, we have that $p'q-(q'+1)p=-(p-1)$ and $p'q-(q'-1)p=p+1$. Thus, if $d$ divides any two elements of $\{q, q'+1,p-1\}$, it must divide the third. Similarly, if $d$ divides any two elements of $\{q, q'-1, p+1\}$, it must divide the third. The result follows.\end{proof}

\begin{lem} $\frac{p-p'}{q-q'}=[a_1,...,a_n-1]$\label{lem3}\end{lem}

\begin{proof} We will prove this by induction on $q$. First, let $q=2$ and $p>2$ is odd. Then $\frac{p}{2}=[\frac{p+1}{2},2]$ and $\frac{p'}{q'}=\frac{p+1}{2}$ (and in particular, $q'=1$). Then $\frac{p-p'}{q-q'}=\frac{p-1}{2}=[\frac{p+1}{2},1]$. Now, assume the result is true for all fractions satisfying $q\le k-1$. Let $\frac{p}{k}=[a_1,...,a_n]$ so that $\frac{p'}{k'}=[a_1,...,a_{n-1}]$. Furthermore, let $t$ and $t'$ be integers such that $\frac{k}{t}=[a_2,...,a_n]$ and $\frac{k'}{t'}=[a_2,...,a_{n-1}]$. Then by the inductive hypothesis, $\frac{k-k'}{t-t'}=[a_2,...,a_n-1]$. Now, $p=a_1k-t$ and $p'=a_1k'-t'$. Thus, $\frac{p-p'}{k-k'}=\frac{a_1(k-k')-(t-t')}{k-k'}=a_1-\frac{t-t'}{k-k'}=[a_1,...,a_n-1].$\end{proof}

\begin{lem} $\frac{p-q}{p'-q'}=[a_n,...,a_1-1]$\label{lem4}\end{lem}

\begin{proof} By Lemma \ref{lem4}, we have $\frac{p-p'}{q-q'}=[a_1,...,a_n-1]$. Thus, $[a_1-1,...,a_n-1]=\frac{p-p'}{q-q'}-1=\frac{p-p'-q+q'}{q-q'}$. Let $\frac{p-p'-q+q'}{k}=[a_n-1,...,a_1-1]$, where $k$ is the unique integer satisfying $1<k<p-p'-q+q'$ and $k(q-q')\equiv1\mod(p-p'-q+q')$. We claim $k=p'-q'$. Indeed, $(p'-q')(q-q')=(p-p'-q+q')q'+1$ (since $pq'+1=p'q$). Thus $[a_n,...,a_1-1] = 1+\frac{p-p'-q+q'}{p'-q'}=\frac{p-q}{p'-q'}$.\end{proof}

\bibliographystyle{plain}
\bibliography{Bibliography}

\end{document}